\documentclass[11pt,reqno, a4paper]{amsart}
\usepackage[margin=1in]{geometry}

\usepackage{graphicx}
\usepackage{amsthm, amsfonts,amssymb} 
\usepackage{lmodern}
\usepackage[T1]{fontenc}
\usepackage{mathtools}
\newtheorem{thm}{Theorem}[section]
\newtheorem{lem}[thm]{Lemma}

\newtheorem{cor}[thm]{Corollary}

\newtheorem{mydef}{Definition}[section]

\newtheorem{rem}{Remark}[section]
\theoremstyle{remark}
\newcommand\bb[1]{\mathbf{#1}}
\newcommand\ddfrac[2]{\frac{\displaystyle #1}{\displaystyle #2}}

\DeclareMathOperator*{\esssup}{ess\,sup}

\usepackage[font=footnotesize,labelfont=bf]{caption}
\usepackage{mathtools}

\usepackage[utf8]{inputenc}
\date{}

\begin{document}

\title{Contact in fluid-plate interaction: formation and detachment}
\author[Sr\dj{}an Trifunovi\'c]{Sr\dj{}an Trifunovi\'c$^1$}

\thanks{$^1$Department of Mathematics and Informatics, Faculty of Sciences, University of Novi Sad, Serbia. 
e-mail: srdjan.trifunovic@dmi.uns.ac.rs}

\begin{abstract}
	In this paper, we study an interaction problem between an elastic plate and a compressible viscous fluid located between the rigid bottom $z=0$ and the plate. First, by utilizing the vertical fluid dissipation, we show that $\ln\eta(t) \in L^1$ for any $t>0$ provided that $\ln\eta_0\in L^1$, ensuring that additional plate contact can form only on a set of a measure zero. Then, by utilizing the expanding capability of compressible fluid pressure, we show that all contact has to detach in finite time provided that the source force acting onto the plate is not pushing down excessively. Finally, we show that contact at any point can be detached in any given time with a strong enough source force localized around that point which is pulling the plate up. The results are based on a novel variational inequality which preserves key information about the total force applied onto the plate, even in presence of contact. This is the first result where detachment of contact is proven in fluid-structure interaction.
\end{abstract}

\maketitle

\bigskip
\begin{center}
    \large{\textit{Dedicated to \v{S}\'{a}rka Ne\v{c}asov\'{a} for her 60th birthday}}
\end{center}
\bigskip
\bigskip

\\${}$ 

\section{The model}
Here, we study an interaction problem between a compressible viscous fluid and an elastic plate. The plate moves only vertically and it is located above the rigid bottom $\{z=0\}$. Therefore, it is described by non-negative scalar displacement $\eta:G_T\to \mathbb{R}_0^+$, where $G_T:=(0,T)\times \Gamma$ and $\Gamma$ is a 2D torus. The displacement $\eta$ satisfies the linear plate equation on $G_T$:
	\begin{eqnarray}
		\eta_{tt}+ \Delta^2 \eta =-S^\eta \bb{f}_{fl}\cdot \mathbf{e}_z + F, \label{structureeqs}
	\end{eqnarray}
where $\bb{f}_{fl}$ is the force with which fluid acts onto the plate, $\bb{e}_z:=(0,0,1)$, $S^\eta:=\sqrt{1+|\nabla\eta|^2}$ is the Jacobian of the transformation from the Eulerian to the Lagrangian coordinates of the plate, and $F$ is the outer force. The graph 	of $\eta(t)$ is denoted as
    \begin{eqnarray*}
 \Gamma^\eta(t):= \{(x,y,\eta(t,x,y)): (x,y)\in \Gamma \}, \quad t\in [0,T].
    \end{eqnarray*}
The fluid fills the domain at time $t$ (see figure $\ref{fig2}$)
	\begin{eqnarray*}
		\Omega^\eta(t):=\{(x,y,z): (x,y)\in \Gamma, 0<z<\eta(t,x,y)\}, \quad t\in [0,T],
	\end{eqnarray*}
while the time-space cylinder is denoted as
	\begin{eqnarray*}
		Q_T^\eta:= \bigcup_{t\in (0,T)} \{t\} \times \Omega^\eta(t).
	\end{eqnarray*}
The fluid is described by the density $\rho:Q_T^\eta\to \mathbb{R}_0^+$ and velocity $\bb{u}=(u_1,u_2,u_3):Q_T^\eta\to \mathbb{R}^3$ and it is governed by the compressible Navier-Stokes equations on $Q_T^\eta$:
	\begin{equation}
		\begin{split}
              \partial_t \rho + \nabla \cdot (\rho \bb{u}) &= 0, \\
              \partial_t (\rho\bb{u}) + \nabla \cdot (\rho\bb{u}\otimes \bb{u})& = -\nabla p(\rho) +\nabla \cdot \mathbb{S}(\nabla \bb{u}), 
		\end{split} 
	\end{equation}
where 
\begin{eqnarray*}
    \mathbb{S}(\nabla \bb{u}):=\mu \left( \nabla \bb{u} + \nabla^\tau \bb{u}- \frac23\nabla \cdot \bb{u} \mathbb{I} \right) + \lambda \nabla \cdot \bb{u}\mathbb{I},\quad \mu>0, ~\lambda \geq 0,
\end{eqnarray*}
$\mathbb{I}$ is the $3\times 3$ identity matrix, and the pressure $p$ is set to be
	\begin{equation*}
		p(\rho)=\rho^\gamma.
	\end{equation*}
The fluid velocity satisfies the no-slip condition on the rigid bottom $\{z=0\}$:
\begin{eqnarray}
    \bb{u}(t,x,y,0) = 0, \quad t\in(0,T),~ (x,y)\in \Gamma.
\end{eqnarray}
Finally, the following kinematic and dynamic coupling conditions are assumed on $G_T$:
\begin{eqnarray}
    \partial_t \eta(t,x,y) \bb{e}_z&=& \bb{u}(t,x,y,\eta(t,x,y)),\label{kinc}\\
    \bb{f}_{fl}(t,x,y)&=&\big[(-p(\rho)\mathbb{I}+ \mathbb{S}(\nabla \bb{u}))\nu^\eta\big](t,x,y,\eta(t,x,y)), \label{dync}
\end{eqnarray}
where $\nu^\eta=(-\nabla \eta,1)/\sqrt{1+{|\nabla\eta|^2}}$ is the unit outward normal vector on the graph $\Gamma^\eta$.

\begin{figure}[h!]
		\centering
		\includegraphics[width=0.85\linewidth]{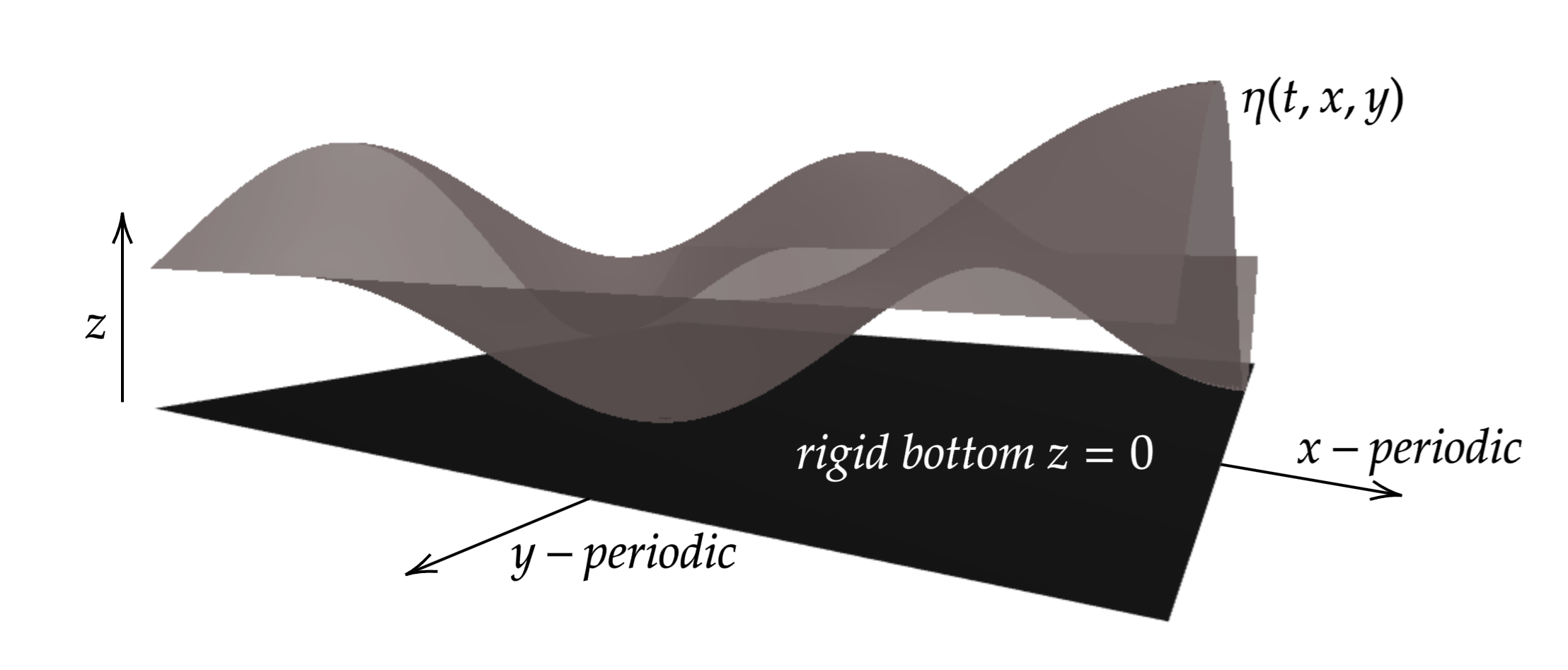}
		\caption{An example of the fluid domain $\Omega^\eta(t)$ located between the rigid bottom $z=0$ and elastic plate $z=\eta(t,x,y)$.}
		\label{fig2}
\end{figure}

\section{Weak formulation and main results}
First, the problem at hand allows contact which a priori induces irregular shapes of the fluid domain and consequently we cannot expect solutions to be regular. Indeed, this has been observed in various models (see for example \cite{star}). Moreover, any sufficiently regular solution cannot have contact as it was shown in \cite{BreitContact}. Therefore, the only reasonable framework is the one of weak solutions. We start with introducing some concepts well-known in fluid-structure interaction literature.

\bigskip

First, since the fluid domain can degenerate into multiple disconnected regions at any time (even possibly into continuum many of them as it was observed in \cite[Section~3.1]{MalteSrdjan}), some notions which are standard in fluid-structure interaction analysis have to be carefully redefined. Therefore, for a given $\eta(t) \in H^2(\Gamma)$ such that $\eta(t)>0$ a.e. on $\Gamma$, we start by introducing:
    \begin{eqnarray*}
        H_{0z}^1(\Omega^\eta(t)):= \{ f\in H^1(\Omega^\eta(t)): \gamma_{|O\times\{0\}}f=0 \text{ for every closed set } O \subset \{\eta(t)>0\} \},
    \end{eqnarray*}
where $H^1(\Omega^\eta(t)) \equiv W^{1,1}(\Omega^\eta(t))$ is the standard Sobolev space, $\gamma_{|O\times\{0\}}f$ is the trace of $f$ on $O\times\{0\}$, and $\{\eta(t)>0\}:=\{(x,y)\in \Gamma: \eta(t,x,y)>0\}$. Next, given $f\in H_{0z}^1(\Omega^\eta(t))$, the Lagrangian trace operator is defined as
    \begin{eqnarray*}
        \gamma_{|\Gamma^\eta(t)} f:= \hat{\gamma}_{|\Gamma^\eta(t)} e_0[f],
    \end{eqnarray*}
where $e_0[f]$ is the extension by zero to $\Gamma\times (-1,0)$ which preserves the norm $\| f\|_{H^1(\Omega^\eta(t))}=\|e_0[f]\|_{H^1(\Omega^\eta(t)\cup \Gamma\times (-1,0))}$ by Lemma $\ref{ext:zero}$, while $\hat{\gamma}_{|\Gamma^\eta(t)}$ is the standard Lagrangian trace defined as 
    \begin{eqnarray*}
        \hat{\gamma}_{|\Gamma^\eta(t)} g (x,y) := g(x,y,\eta(t,x,y)), \quad (x,y) \in \Gamma,
    \end{eqnarray*}
for any $g\in C^1(\overline{\Omega^\eta(t)\cup \Gamma\times(-1,0) })$ and then extended to a continuous linear operator from $H^1(\Omega^\eta(t)\cup \Gamma\times (-1,0))$ to $H^s(\Gamma)$ for any $s\in (0,1/2)$ (see \cite{boristrace}). For time-dependent functions, trace will be denoted as
    \begin{eqnarray*}
        \left(\gamma_{|\Gamma^\eta} f\right)(t)=\gamma_{|\Gamma^\eta(t)} f(t).
    \end{eqnarray*}
Finally, we define the standard Bochner spaces corresponding to the moving domains:
\begin{eqnarray*}
    L^p(0,T; L^q(\Omega^\eta(t))):= &&\{ f\in L^1(Q_T^\eta):\\ &&f(t) \in L^q(\Omega^\eta(t)) \text{ for a.a }t\in(0,T),~ \|f(t)\|_{L^q(\Omega^\eta(t))}\in L^p(0,T)\}, 
\end{eqnarray*}
with the norms (note that "$(t)$" is omitted in the norm for simplicity)
    \begin{eqnarray*}
         \|f\|_{L^p(0,T; L^q(\Omega^\eta))}&:=& \left(\int_0^T \|f(t)\|_{L^q(\Omega^\eta(t))}^p\right)^{\frac1p}, \quad  p\in[1,\infty), \\
        \|f\|_{L^\infty(0,T; L^q(\Omega^\eta))}&:=& \esssup_{t\in (0,T)} \|f(t)\|_{L^q(\Omega^\eta(t))},
    \end{eqnarray*}
and
    \begin{eqnarray*}
        L^p(0,T; W^{1,q}(\Omega^\eta(t))):= \{ f\in L^p(0,T; L^q(\Omega^\eta(t))): \nabla f \in L^p(0,T; L^q(\Omega^\eta(t))) \}, 
    \end{eqnarray*}
with the norms
    \begin{eqnarray*}
        \|f\|_{L^p(0,T; W^{1,q}(\Omega^\eta))}&:=& \left(\int_0^T \|f(t)\|_{W^{1,q}(\Omega^\eta(t))}^p \right)^{\frac1p}, \quad p\in[1,\infty), \\
         \|f\|_{L^\infty(0,T; W^{1,q}(\Omega^\eta))}&:=& \text{ess}\sup_{t\in (0,T)} \|f(t)\|_{W^{1,q}(\Omega^\eta(t))}.
    \end{eqnarray*}
Specially, the following space is defined
    \begin{eqnarray*}
        L^p(0,T; H_{0z}^1(\Omega^\eta(t))):= \{ f\in L^p(0,T; H^1(\Omega^\eta)): f(t) \in H_{0z}^1(\Omega^\eta(t)) \text{ for a.a. }t\in(0,T) \}.
    \end{eqnarray*}

\bigskip

The regularity and compatibility conditions for initial data are as follows
    \begin{eqnarray}
    \begin{cases}
         &\eta_0 >0 \text{ a.e. on } \Gamma ,\quad \eta_0 \in H^2(\Gamma), \quad \ln\eta_0 \in L^1(\Gamma),\quad v_0 \in L^2(\Gamma), \\
	   &\rho_0\in L^{\gamma}(\Omega^{\eta_0}), \quad (\rho\bb{u})_0 \in L^{\frac{2\gamma}{\gamma+1}}(\Omega^{\eta_0}) , \quad  \frac{(\rho \bb{u})_0^2}{\rho_0} \in L^1(\Omega^{\eta_0}),\\ 
	   &\rho_0>0 \text{ in } \{(x,y,z)\in \Omega^{\eta_0}: (\rho\bb{u})_0(x,y,z) >0\},\quad \int_{\Omega^{\eta_0}}\rho_0 := m>0 ,\quad \rho_0 \geq 0,   \end{cases}  \label{str:id} 
    \end{eqnarray}
where $\Omega^{\eta_0}:=\Omega^\eta(0)$ and $m$ is the total mass. Note that there is a new condition on $\ln\eta_0$ which allows contact initially. Let us now introduce a concept of weak solution which will be studied:

\begin{mydef}[\textbf{Weak solution}]\label{weak:sol:def}
We say that $(\rho,\bb{u},\eta)$ is a weak solution to \eqref{structureeqs}-\eqref{dync} if the initial data satisfies $\eqref{str:id}$, $\gamma>3$ and:
\begin{enumerate}
    \item $\int_{\Omega^\eta(t)}\rho(t)=m$ and $\eta(t)>0$ a.e. on $\Gamma$ for all $t\in[0,T]$, and the solution is in the following regularity class:
      \begin{eqnarray*}
  	&&\rho \in L^\infty(0,T; L^\gamma(\Omega^\eta(t))), \quad        \bb{u}\in L^2(0,T; H_{0z}^1(\Omega^\eta(t))),\\ 
  	 &&\rho |\bb{u}|^2 \in L^\infty(0,T;L^1(\Omega^\eta(t))),    \quad \frac{\rho^\gamma}{\eta}, ~\frac{\rho|u_3|^2}{\eta}\in L^1(Q_T^\eta),\\
         &&\eta \in W^{1,\infty}(0,T; L^2(\Gamma))\cap L^\infty(0,T; H^2(\Gamma)), \quad  \ln\eta \in C([0,T]; L^1(\Gamma));
    \end{eqnarray*}
\item The kinematic coupling $\gamma_{|{\Gamma}^\eta}\bb{u} = \partial_t \eta \bb{e}_z$ holds a.e. on $G_T$;
\item The renormalized continuity equation
    \begin{equation}\label{reconteqweak}
  	\int_{Q_T^\eta} \rho B(\rho)( \partial_t \varphi +\bb{u}\cdot \nabla  \varphi) =\int_{Q_T^\eta} b(\rho)(\nabla\cdot \bb{u}) \varphi - \int_{\Omega^{\eta_0}}\rho_0B(\rho_0) \varphi(0)
	\end{equation}
holds for all functions $\varphi \in C_c^\infty([0,T)\times \mathbb{R}^3)$ and any $b\in L^\infty (0,\infty) \cap C[0,\infty)$ such that $b(0)=0$ with $B(\rho)$ being any primitive function to $\frac{b(\rho)}{\rho^2}$; 
\item The coupled momentum equation
  \begin{eqnarray}
        &&\int_{Q_T^\eta} \rho \bb{u} \cdot\partial_t                     \boldsymbol\varphi + \int_{Q_T^\eta}(\rho \bb{u} \otimes        \bb{u}):\nabla\boldsymbol\varphi +\int_{Q_T^\eta} \rho^\gamma   (\nabla \cdot \boldsymbol\varphi)-  \int_{Q_T^\eta} \mathbb{S}(     \nabla\bb{u}): \nabla \boldsymbol\varphi \nonumber\\
        &&+\int_{G_T} \partial_t \eta \partial_t \psi  - \int_{G_T}\Delta\eta \Delta\psi +\int_{G_T} F\psi \nonumber\\
        &&=  -\int_{\Omega^{\eta_0}}(\rho\bb{u})_0\cdot\boldsymbol\varphi(0) - \int_\Gamma v_0 \psi \label{momeqweak}
\end{eqnarray}
holds for all $\boldsymbol\varphi \in C_c^\infty([0,T)\times \mathbb{R}^3)$ and all $\psi\in C_c^\infty([0,T)\times \Gamma)$ such that $\boldsymbol\varphi(t,x,y,\eta(t,x))=\psi(t,x,y)\bb{e}_z$ on $G_T$ and $\boldsymbol\varphi=0$ on $\{z=0\}$;
\item The contact inequality holds for a.a. $t\in(0,T)$:
    \begin{eqnarray}
        &&\int_0^t\int_{\Gamma} F +\int_0^t \int_{\Omega^\eta}  \frac{\rho^\gamma}\eta +\int_0^t \int_{\Omega^\eta} \frac{\rho |u_3|^2}\eta - \left(\frac{4\mu}3 +\lambda \right)\int_{\Gamma} \ln\eta(t)  \nonumber\\
        &&\leq \int_0^t \int_{\Omega^\eta} \rho u_3 z \frac{\partial_t \eta}{\eta^2}  +\int_0^t \int_{\Omega^\eta} \rho u_1 u_3 \frac{z\partial_x\eta}{\eta^2} + \int_0^t \int_{\Omega^\eta} \rho u_2 u_3\frac{z \partial_y\eta}{\eta^2}\nonumber \\
        &&\quad -\mu\int_0^t \int_{\Omega^\eta } (\partial_z u_1 + \partial_x u_3) \frac{z\partial_x \eta}{\eta^2}  -\mu\int_0^t \int_{\Omega^\eta } (\partial_z u_2 + \partial_y u_3) \frac{z \partial_y \eta}{\eta^2} \nonumber \\
        &&\quad +\left(\lambda - \frac{2\mu}3\right)\int_{Q_T^\eta} \left( u_1 \frac{ \partial_x \eta}{\eta^2}+   
        u_2 \frac{\partial_y \eta}{\eta^2}    \right) - \left(\frac{4\mu}3 +\lambda \right)\int_{\Gamma} \ln \eta_{0} \nonumber \\
        &&\quad+\int_{\Omega^\eta(t)}\rho(t)u_3(t) \frac{z}{\eta(t)} -\int_{\Omega^{\eta_{0}}}(\rho u_3)_0\frac{z}{\eta_{0}} +\int_\Gamma \partial_t \eta(t)-\int_\Gamma v_0;  \label{ver:ineq}
    \end{eqnarray}
\item The energy inequality holds for a.a. $t\in (0,T)$
    \begin{eqnarray}
        &&\displaystyle{\int_{\Omega^\eta(t)} \left( \frac{1}{2} \rho       |\bb{u}|^2 + \frac{\rho^\gamma}{\gamma-1}  \right)(t) +  \int_{\Gamma} \left(\frac{1}{2}|\partial_t \eta|^2 + \frac{1}{2} |\Delta \eta|^2 \right)(t)} + \int_0^t \int_{\Omega^\eta} \mathbb{S}(\nabla \bb{u}):\nabla \bb{u} \nonumber \\ 
        &&  \leq \displaystyle{\int_{\Omega^{\eta_0}} \left( \frac{1}{2\rho_0} |(\rho\bb{u})_0|^2 + \frac{\rho_0^\gamma}{\gamma-1}   \right)} + \int_{\Gamma} \left(\frac{1}{2}|v_0|^2 + \frac{1}{2} |\Delta \eta_0|^2  \right)+ \int_0^t \int_\Gamma F \partial_t \eta .\label{en:ineq} 
    \end{eqnarray}
\end{enumerate}
\end{mydef}

\begin{rem}
(1) The inequality $\eqref{ver:ineq}$ in this definition is new. It can formally be obtained from $\eqref{momeqweak}$ by choosing $\boldsymbol\varphi=\chi_{\{0,t \}}\frac{z}{\eta}\bb{e}_z$ and $\psi =\chi_{\{0,t \}}$, at least provided that $\eta>0$, as it was done in Section \ref{cont:sec}. It preserves crucial information about total forces acting onto the plate, even if $\{\eta = 0\}\neq \emptyset$, which is in contrast with $\eqref{momeqweak}$ where this information is lost. The inequality, instead of equality, comes from possible defects of $\frac{\rho^\gamma}\eta$, $\partial_t(\ln\eta^-)$ and the contact force which is a result of plate not being able to pass through the rigid boundary $\{z=0\}$. For some applications, it is worth noting that it can be localized in time-space, i.e. it holds in the form of $\eqref{eq:contact}$ as an inequality for all non-negative $\psi \in C_c^\infty([0,T)\times \Gamma)$ such that $\partial_t \psi\leq 0$.\\
(2) Without contact, weak solutions for similar problems were constructed in \cite{Breit,trwa2}, strong solutions in \cite{MRT,MT,mitra} and weak-strong uniqueness was shown in \cite{WSU}.
\end{rem}

The first result considers existence of weak solutions:	
\begin{thm}[\textbf{Global existence}]\label{main1}
Let initial data satisfy $\eqref{str:id}$, let $T>0$, $\gamma>3$ and $F\in L^1(0,T; L^2(\Gamma))$. Then, there exists a weak solution $(\rho,\bb{u},\eta)$ in the sense of Definition $\ref{weak:sol:def}$ on $(0,T)$.
\end{thm}

The second result considers the detachment of all contact as a consequence of pressure pushing up the plate. As time goes on, all effects depending on fluid velocity (which includes the structure velocity due to kinematic coupling) will dissipate and all that will remain will be plate bending, pressure and source force. As the total bending vanishes $\int_\Gamma \Delta^2 \eta=0$ due to periodic boundary conditions,  pressure will be able to push the plate up and break the contact if the outer force $F$ is not pushing down too much. In particular, if $F\in L^1(0,\infty; L^2(\Gamma))$, then $F$ will lose its strength over time resulting in pressure pushing up the structure and breaking all the contact unconditionally, since the total pressure is bounded from below due to fixed mass. On the other hand, it might be of independent interest to observe a time-independent force $F\in L^1(\Gamma)$ which models for example atmospheric pressure or gravity acting (downwards) onto the plate. In this case, as time goes on this force will remain and the detachment will happen if the total outer force $\int_\Gamma F$ is not too negative. Therefore, the following result is split into two cases:

\begin{thm}[\textbf{Detachment of all contact by pressure}]\label{main2}
Let initial data satisfy $\eqref{str:id}$, $\gamma>3$ and let $(\rho,\bb{u},\eta)$ be a weak solution in the sense of Definition $\ref{weak:sol:def}$ defined on $(0,\infty)$ corresponding to a given force $F$. Then:
\begin{enumerate}
    \item If $F\in L^1(0,\infty; L^2(\Gamma))$, then there exists a $T>0$ such that $\eta(T)>0$;
    \item If $F\in L^1(\Gamma)$ such that\footnote{The existence of a weak solution in this case follows in the same way by noticing that the term $\int_0^t\int_\Gamma F \partial_t \eta$ in the energy inequality $\eqref{en:ineq}$ can be controlled as in $\eqref{l1:case}$.} $\int_\Gamma F>-A$ with $A>0$ being small enough with respect to initial data and $\Gamma$, then then there exists a $T>0$ such that $\eta(T)>0$.
\end{enumerate}
\end{thm}

The last result considers the detachment of contact at a point by a strong localized force pulling up. In a sense, it can be described in terms of control theory, i.e. we show that our system is controllable from a state with contact to a state without contact for any given time by constructing an explicit control which achieves this, however without a precise target state but rather a family of target states. The proof is based on a contradiction argument. Namely, if we assume that the contact cannot be broken, then the source force can be localized and increased arbitrarily  near that point without adding almost any energy to the system, as the plate is not able to move too much near contact point. This is however in contradiction with the contact inequality $\eqref{ver:ineq}$, as increasing the source force $F$ without increasing the energy will increase the left-hand side of $\eqref{ver:ineq}$ and keep the right-hand side bounded.

\begin{thm}[\textbf{Detachment of contact at a point by outer force}]\label{main3}
Let initial data satisfy $\eqref{str:id}$, $\gamma>3$ and let $\eta_0(x_0,y_0)=0$ for some $(x_0,y_0)\in \Gamma$. Then, for any $T>0$, there exists a force $F\in L^1(\Gamma)$ with $F\geq 0$ such that for any weak solution $(\rho,\bb{u},\eta)$ in the sense of Definition $\ref{weak:sol:def}$ defined on $(0,T)$, one has $\eta(t,x_0,y_0)>0$ for some $t\in(0,T)$.
\end{thm}

It is worth noting that the same results can be obtained in a 2D/1D case where plate is replaced by a beam, by using better Sobolev imbedding results for $\bb{u}$ and thus lowering $\gamma$:
\begin{cor}
The same results obtained in Theorem  $\ref{main1}$,  Theorem $\ref{main3}$ and  Theorem $\ref{main2}$ hold in the 2D/1D case with $\gamma>2$.
\end{cor}

\section{Literature review and discussion}
The fluid-structure interaction problems have been widely studied in the mathematical literature for more than two decades. Their motivation lies in various applications ranging from industry, aerodynamics, biomedicine and so on (see \cite{book,book2} and references therein). The existence theory for fluid-plate/shell interaction including incompressible fluids has been so far well-developed, starting from the work of H.~Beirão da Veiga \cite{veiga} and Grandmont et al \cite{grandmont3,grandmont4}, continuing with the work of Lequerre \cite{Leq1},
Lengeler and R{\r u}{\v z}i{\v c}ka \cite{LeRu14}, Muha and {\v C}ani{\'c} \cite{BorSun,BorSun2}, Muha and  Schwarzacher \cite{BS}, and to more recent results concerning regularity and uniqueness of 2D/1D interaction system \cite{BreitSolo, BMSS, SchSu}, vectorial structures in 2D/1D case \cite{unrestricted,GS} and 3D/2D case \cite{BorSunTaw}, and inviscid incompressible fluid and plate interaction \cite{BKMT}. The problem of compressible viscous fluids and plate/shell interaction has been studied far less. It was started by Breit and Schwarzacher \cite{Breit}, followed by Mitra \cite{mitra}, Maity, Roy and Takahashi \cite{MRT}, Trifunovi\'{c} and Wang \cite{trwa2}, then involving heat conductivity in the fluid in \cite{BS2,MT}, and finally the problem where both fluid and plate conduct and exchange heat \cite{heatexchange}.

The existence theory has been for most limited by occurrence of contact. Indeed, since contact induces fluid domain degeneration and irregularities, most of solutions constructed in literature fail to "survive" it and usually the solution is studied up to the first moment when it happens. There have been two main approaches in dealing with this issue. First one is showing that contact will not occur and then obtaining a global solution. This was done in the context of incompressible viscous fluid and rigid body interaction in \cite{collisions2d,collisions3d} and then later for the interaction between incompressible viscous fluid and viscoelastic beam in \cite{global}. The second approach is allowing for contact to happen but with test functions vanishing in the contact region, thus losing information about the post-contact dynamics. This has been the case in fluid and rigid body interaction (see for example \cite{sarka} and the references therein) where this allows two outcomes -- sticking to boundary or going away from it \cite{star}. In the context of elastic structures and incompressible viscous fluids, first such result was in \cite{CGH} where the solution was obtained as a limit of global ones constructed in \cite{global} as viscosity in the beam vanishes. Later in \cite{MalteSrdjan}, a weak solution to the interaction between 3D viscoelastic body and a 3D compressible viscous fluid with allowed contact and pressure defects was constructed by penalization of interpenetration of the body and its contact with the rigid boundary. 

While it has been shown that contact occurs in some results concerning rigid bodies \cite{regularitycollisions,JNOR}, there were no results that were able to answer what happens afterwards since the fluid tends to behave irregularly near the contact, making it difficult to preserve any information. The inequality $\eqref{ver:ineq}$ is the first one that stores crucial information as it holds even when there is contact. It is reminiscent of variational inequalities in the contact problems for plates \cite{varineq}. First of all, it tells us about the boundedness of $\ln\eta$ which comes from the vertical fluid dissipation and kinematic coupling. This bound in a sense is very weak, but it allows a large degree of freedom for the occurrence of contact, especially for the initial data. In particular, it is also weaker than the one on $\frac1\eta$ obtained in the case of 2D incompressible viscous fluid interacting with 1D viscoelastic beam in \cite[Proposition 9]{global}. This is to be expected as here there is no viscoelasticity, the structure is 2D and the fluid is 3D and compressible, which all together drastically reduce the amount of regularity in our system. Second important information coming from the inequality $\eqref{ver:ineq}$ is about the total pressure exerted onto the plate, even in the region near the contact. Indeed, if $\rho$ is regular enough, then
\begin{eqnarray*}
    \int_{\Omega^\eta} \frac{p(\rho)}\eta = \int_{\Gamma^\eta} p(\rho) - \int_{\Omega^\eta} \partial_z p(\rho) \frac{z}\eta
\end{eqnarray*}
where the last term is then expressed through other fluid terms and then dealt with. Since the total bending $\int_\Gamma \Delta^2\eta$ vanishes due to periodic boundary conditions and since effects depending on fluid velocity dissipate through time, at large time we will have that $\rho,\eta$ are close to a constant state so
\begin{eqnarray*}
    \left(\frac{m}{|\Gamma| ~\eta}\right)^\gamma \sim p(\rho) \sim -F,
\end{eqnarray*}
where $m$ is the mass of the fluid. Thus, if $F$ is small then $\eta$ has to be large and there cannot be any contact. In other words, the pressure will keep pushing up the plate until the fluid is rare enough to balance out with the small force $F$ of course if $\int_\Gamma F< 0$, otherwise the plate will go up indefinitely. Finally, let us also point out that the new estimate on $\frac{\rho^\gamma}\eta$ in $L^1(Q_T^\eta)$ coming from the contact inequality $\eqref{ver:ineq}$ gives us certain better control of pressure near contact. In particular, it ensures that no concentrations (measures) occur throughout the convergence of approximate solutions as it was proven in Lemma $\ref{weak:conv:pr:delta}$, even though the fluid domain near contact can be very irregular. In a sense, it tells us that plate regularizes the pressure near contact.\\

Let us point out some observations about other directions:
\begin{enumerate}
    \item \textbf{Adiabatic constant} $\gamma>3$. This is by far the most restricting condition in the presented results and it fails to include most of the physically reasonable models which are in the range $\gamma \in (1, 5/3]$. While the detachment of contact essentially relies on the strength of the outer force or on the boundedness of the average density from below (i.e. on fixed mass), the other terms on the right-hand side of $\eqref{ver:ineq}$ cannot be controlled without enough integrability of $\rho$. Thus, lowering $\gamma$ remains an open problem;
    \item \textbf{Incompressible case}. While this case is easier to study in most of scenarios of fluid-structure interaction problems, the detachment of contact seems not to be one of them. The issue lies in the incompressible fluid pressure and its lack of a sign, regularity and its inability to do work ($p\nabla\cdot\bb{u}$ = 0). Moreover, the divergence-free condition on $\bb{u}$ makes it impossible to decouple vertical and horizontal fluid effects, so obtaining the contact inequality $\eqref{ver:ineq}$ seems to be out of reach, as this inequality gives us information only about the vertical fluid effects applied onto the plate. This can also be seen by including the Mach number in our fluid system, which would then appear in the numerator in the right-hand side of $\eqref{blowup}$, giving us only non-negativity of $\eta$ as Mach number goes to zero. 
    \item \textbf{Bounded plate domain}. Let us point out that it is possible to prove Theorem $\ref{main1}$ and Theorem $\ref{main3}$ in the case where $\Gamma$ is a bounded Lipschitz domain and the plate is fixed on the boundary $\partial\Gamma$ at the height $H>0$, although with some additional technical steps in the existence proof in order to deal with the domain corners, i.e. regularization of domain should be performed. In that case, the contact inequality is then localized away from the boundary in the from of $\eqref{eq:contact}$, since the plate cannot have contact near the boundary $\partial\Gamma$ due to uniform estimates. However, it seems that proving Theorem $\ref{main2}$ in this case is not obvious, as the total bending on a bounded domain does not vanish, i.e. $\int_\Gamma \Delta^2 \eta = \int_{\partial\Gamma} \nabla \Delta \eta \cdot \nu \neq 0$, where $\nu$ is the outer normal vector on $\partial\Gamma$. As time goes to infinity, total bending will simply balance out the total pressure and outer force, so no new information is obtained.
    \item \textbf{Source forces}. (a) Let us first point out that the regularity of $F\in L^1(0,T; L^2(\Gamma))$ is not the only option for the existence, but it is the simplest one. Indeed since $\partial_t \eta \in L^2(0,T; H^s(\Gamma))$ by the trace regularity given in Lemma $\ref{trace:lemma}$, it is also enough to assume that $F\in L^2(0,T; H^{-s}(\Gamma))$ for any $s\in(0,1/2)$. In fact, any option that allows us to control the term $\int_0^t\int_\Gamma F\partial_t \eta$ in the energy inequality is enough, as it can be seen in $\eqref{l1:case}$ for example. \\
    (b) Another source force acting onto the fluid of the form $\rho \bb{G}$ can be added, with $\bb{G}$ satisfying certain properties, however it was not included for the simplicity of the presentation.
\end{enumerate}

\section{Construction of global weak solutions -- proof of Theorem \ref{main1}}
The multi-layered construction in this section based on decoupling, domain extension and two-side penalization comes from \cite{heatexchange}, so we omit some of the technical details and present the most important steps. The decoupling approach (Lie operator splitting) combined with time-discretization was first used in \cite{BorSun} where a 2D/1D incompressible viscous fluid and (visco)elastic beam interaction model was studied. It was later extended to a hybrid decoupling approximation scheme in \cite{trwa}, where a nonlinear plate model was dealt with by using a time-continuous equation for the plate and stationary for the fluid, and then to fully time-continuous time-marching decoupling scheme which penalizes the kinematic coupling from both equations \cite{trwa2}, to study a nonlinear thermoelastic plate and a compressible viscous fluid interaction problem. Finally, this method was then combined with a domain extension approach developed in \cite{commoving} to study an interaction problem between a heat-conducing compressible viscous fluid and a thermoelastic shell with heat exchange in \cite{heatexchange}.

\subsection{The approximate problem}
Fix $T>0$ and let $\delta,\varepsilon\in (0,1/2)$ be approximation parameters. Moreover, let $N\in \mathbb{N}$ and $\Delta t := T/N$. The time interval is split into $N$ equal intervals of length $\Delta t$ and a time-marching approach is utilized (there is no discretization in time). The fluid domain is extended to a larger domain (see figure \ref{fig1}) $\Omega_M:=\Gamma\times (0,M)$ where $M>0$ is large enough (which is calculated later in $\eqref{eta:max:est}$) so that the structure will remain at most $M/2$ on $(0,T)$, while we denote $Q_T^M:= (0,T)\times \Omega^M$. The kinematic coupling is penalized on the interface, meaning that the fluid can pass through the elastic boundary. The viscosity $\mathbb{S}(\nabla \bb{u})$ is approximated with $\chi_\eta^\varepsilon\mathbb{S}(\nabla\bb{u})$ where $\varepsilon\leq \chi_\eta^\varepsilon\leq 1$ is smooth and
	\begin{eqnarray}
		\chi_\eta^\varepsilon(t,x,y,z):=
        \begin{cases}
         1,& 0<z<\eta(t,x,y), \\
         \varepsilon,& \eta(t,x,y)+\varepsilon <z<M.
		\end{cases} \label{chi}
	\end{eqnarray}
Additional viscoelastic term $ -\varepsilon\partial_t\Delta \eta$ is added to the plate equation in order to ensure that density will vanish outside of $Q_T^\eta$ once we pass to the limit $\Delta t\to 0$. We approximate the initial displacement as
	\begin{eqnarray*}
        \eta_{0,\delta}:=\eta_0+\delta\geq \delta
	\end{eqnarray*}
and consequently $\Omega^{\eta_{0,\delta}}:=\{(x,y,z): (x,y)\in \Gamma, 0<z<\eta_{0,\delta}(x,y) \}$. We introduce the approximate fluid initial data:
	\begin{eqnarray*}
		&&\rho_{0,\delta}\geq 0,~~ \rho_{0,\delta}\not\equiv 0,~~ \rho_{0,\delta}|_{\Omega^M\setminus \Omega^{\eta_{0,\delta}}}=0, ~~\int_{\Omega_M}\rho_{0,\delta}^{\gamma} \leq c,\\ &&\rho_{0,\delta}\to \rho_0 \text{ in }L^{\gamma}(\Omega_M)~~ \text{ and } ~~ |\{\rho_{0,\delta}<\rho_0 \}|\to 0 \text{ as } \delta \to 0,\\
		&&(\rho\bb{u})_{0,\delta}=\begin{cases} (\rho\bb{u})_0, &\text{ if }\rho_{0,\delta}\geq \rho_0, \\
		0, &\text{ otherwise}\end{cases},\quad \int_{\Omega_M} \frac{1}{\rho_{0,\delta}}|(\rho\bb{u})_{0,\delta}|^2 \leq c.
	\end{eqnarray*}
Finally denote the translation in time by $-\Delta t$ as
	\begin{eqnarray*}
		T_{\Delta t}f(t):=
        \begin{cases}
        f(t-\Delta t),& \quad t\in [(n-1)\Delta t, n\Delta t], ~n\geq 1, \\
        f(0),& \quad t\in [0, \Delta t].
        \end{cases}
	\end{eqnarray*}
\begin{figure}[h!]
		\centering
		\includegraphics[width=0.5\linewidth]{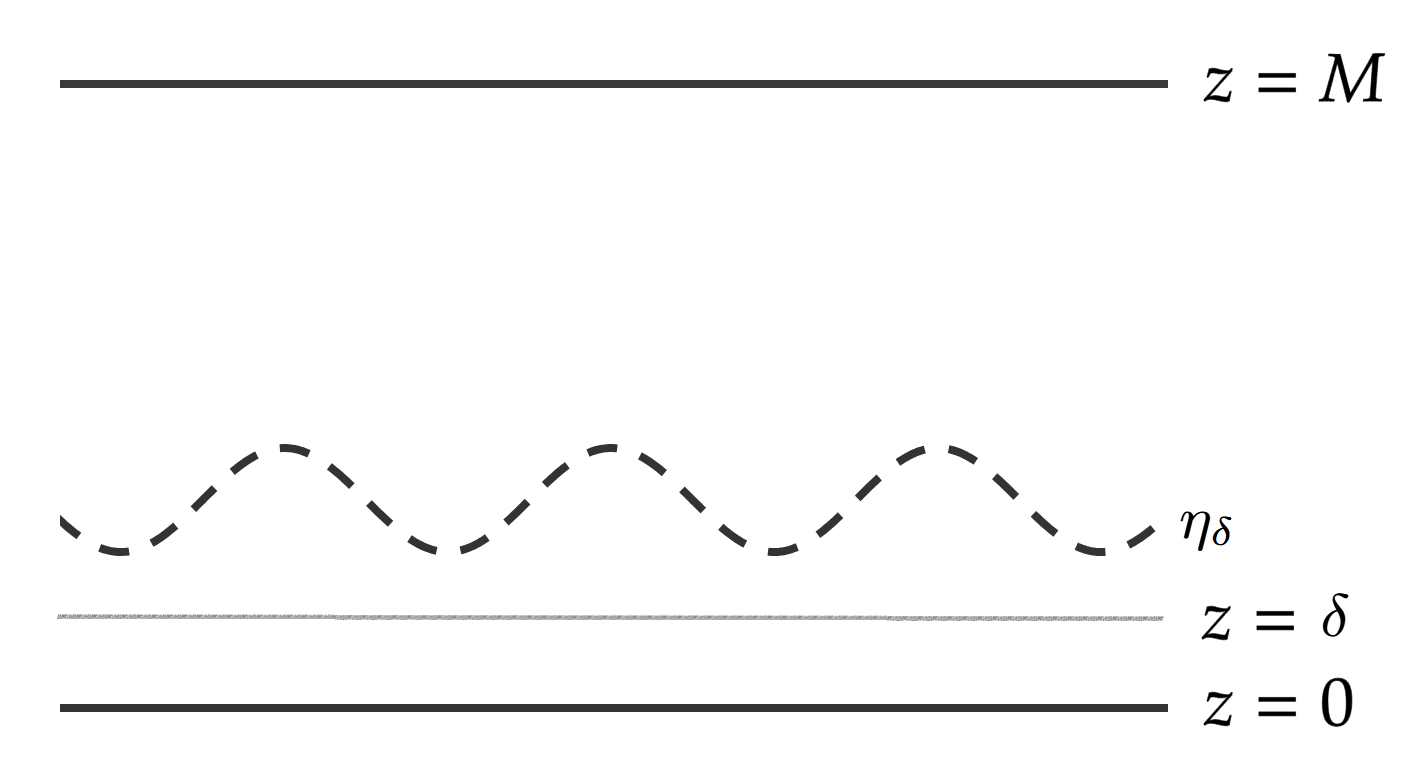}
		\caption{Fluid domain for the approximate solution represented in 2D for simplicity. The graph of $\eta^{n+1}$ is dashed to emphasize that fluid can pass through, as long as $\Delta t>0$. This allows us to construct the solution to (FSP) on a fixed domain, thus avoiding major difficulties in construction which are usually present in fluid-structure interaction problems due to change of the fluid domain in time. Note that there is also an artificial obstacle $z=\delta$ in order to ensure that $\eta^{n+1}\geq \delta$.}
		\label{fig1}
\end{figure}
	
\subsubsection{Structure sub-problem (SSP)}
By induction on $n\geq 0$, assume that:
\begin{itemize}
	\item[] \textbf{Case} $n=0$: $\eta^0(0):=\eta_{0,\delta}$,~ $\tilde{v}^0: = v_0$,~ and
		\begin{eqnarray*}
            \bb{v}^0(t): = v_0\bb{e}_z, \quad \text{for }t\in[-\Delta t,0],
		\end{eqnarray*}
    with
    \begin{eqnarray*}
        E_{str,\varepsilon}^0:=\frac{1-\varepsilon}{2} \int_{\Gamma} |v_0|^2 + \frac{1}{2} \int_{\Gamma}|\Delta \eta_{0,\delta}|^2 ;
    \end{eqnarray*}
\item[] \textbf{Case} $n\geq 1$: the solution $\eta^n$ of $(SSP)$ and the solution $(\rho^n,\bb{u}^n)$ of $(FSP)$ (defined below) are already obtained, and
\begin{eqnarray}  
    \begin{cases}
    &E_{str,\varepsilon}^n:= \lim_{m\to \infty} \left(  \frac{1-\varepsilon}2\int_{\Gamma}|\partial_t \eta^{n}|^2(t_m)
    +\frac12\int_\Gamma |\Delta\eta^n|^2(t_m) \right), \\
    &\int_\Gamma\tilde{v}^n\psi:= \lim_{m\to \infty} \int_\Gamma \partial_t\eta(t_m)\psi \quad \text{ for all } \psi \in L^2(\Gamma), \end{cases} \label{ass:ssp}
\end{eqnarray}
with $((n-1)\Delta t, n\Delta t) \ni t_m\to n\Delta t$ as $m\to \infty$ which are a sequence of Lebesque points of the integrals in the limits chosen so that the limits converge.\\
\end{itemize}

\noindent
Find $\eta^{n+1}$ such that:
\begin{enumerate}
 \item $\eta^{n+1}(n \Delta t) =\eta^{n}(n \Delta t)$;
 \item $\eta^{n+1}\in L^\infty(n\Delta t, (n+1)\Delta t; H^2(\Gamma))\cap W^{1,\infty}(n\Delta t, (n+1)\Delta t; L^2(\Gamma)) \\ \cap H^1(n\Delta t, (n+1)\Delta t; H^1(\Gamma))$;
\item The following plate equation
\begin{eqnarray}
     &&-(1-\varepsilon)\int_{n\Delta t}^{(n+1)\Delta t}\int_{\Gamma}\partial_t \eta^{n+1}\partial_t \psi   +\varepsilon\int_{n\Delta t}^{(n+1)\Delta t}\int_{\Gamma}\ddfrac{\partial_t \eta^{n+1} - T_{\Delta t} \bb{v}^{n}\cdot \bb{e}_z}{\Delta t}\psi \nonumber\\ 
     &&\quad+\varepsilon\int_{n\Delta t}^{(n+1)\Delta t}\int_{\Gamma} \partial_t \nabla \eta^{n+1}\cdot\nabla\psi +\int_{n\Delta t}^{(n+1)\Delta t}\int_{\Gamma}\Delta \eta^{n+1} \Delta \psi \nonumber\\ 
     &&= \int_{n\Delta t}^{(n+1)\Delta t}\int_{\Gamma} F_{con,\delta}^{n+1} - \int_{n\Delta t}^{(n+1)\Delta t}\int_{\Gamma} F^{n+1} \psi  + (1-\varepsilon)\int_\Gamma \tilde{v}^{n}\psi(n\Delta t) \label{plateeqSSP}
\end{eqnarray}
holds for all ${\psi}\in C_c^\infty([n\Delta t,(n+1)\Delta t)\times \Gamma)$, where $F_{con,\delta}^{n+1} \in \mathcal{M}^+([n\Delta t, (n+1)\Delta t]\times \Gamma)$ is the contact force keeping the structure $\eta^{n+1} \geq \delta>0$;
\item The following energy inequality holds for a.a. $t\in (n\Delta t, (n+1)\Delta t)$ 
\begin{eqnarray} 
     && \frac{1-\varepsilon}{2}\int_{\Gamma}|\partial_t \eta^{n+1}|^2(t) +\frac{1}2\int_{\Gamma}|\Delta  \eta^{n+1}|^2(t)\nonumber \\
     &&\quad+\frac\varepsilon{2\Delta t}\int_{n\Delta t}^t\int_{\Gamma}|\partial_t \eta^{n+1} - T_{\Delta t}\bb{v}^{n}\cdot \bb{e}_z|^2+ \frac\varepsilon{2\Delta t}\int_{n\Delta t}^t\int_{\Gamma}|\partial_t \eta^{n+1} |^2 \nonumber\\
     &&\leq \frac\varepsilon{2\Delta t}\int_{n\Delta }^t\int_{\Gamma} |T_{\Delta t}\bb{v}^{n}\cdot \bb{e}_z|^2 +E_{str,\varepsilon}^n.\label{en:ineq:ssp}
\end{eqnarray}
\end{enumerate}

\bigskip

\subsubsection{Fluid sub-problem (FSP)}
By induction on $n\geq 0$, assume that:
	\begin{itemize}
		\item[] \textbf{Case} $n=0$: $\rho^0(0):=\rho_{0,\delta}$,~ $(\rho\bb{u})^0(0): = (\rho\bb{u})_{0,\delta}$ with
        \begin{eqnarray*}
            E_{fl}^0:=\displaystyle{\int_{\Omega^M} \left( \frac{1}{2\rho_{0,\delta}} |(\rho\bb{u})_{0,\delta}|^2 + \frac{\rho_{0,\delta}^\gamma}{\gamma-1}   \right)};
        \end{eqnarray*}
		\item[] \textbf{Case} $n\geq 1$: the solution $(\rho^n,\bb{u}^n)$ of $(FSP)$ and the solution $\eta^{n+1}$ of $(SSP)$ is already obtained, and denote
    \begin{eqnarray}
        E_{fl}^n := \lim_{m\to \infty} \left(\frac12\int_{\Omega^M} \rho^n |\bb{u}^n|^2(t_m) +  \int_{\Omega^M} \frac{(\rho^n)^\gamma}{\gamma-1}(t_m) \right) \label{ass:fsp}
    \end{eqnarray}
with $((n-1)\Delta t, n\Delta t) \ni t_m\to n\Delta t$ as $m\to \infty$ is a sequence of Lebesque points of the integrals in the limit chosen so that the limit converges.\\
\end{itemize}
    
\noindent
Find $(\rho^{n+1}, \bb{u}^{n+1})$ such that:
	\begin{enumerate}
    \item $\rho^{n+1}(n\Delta t)= \rho^{n}(n\Delta t)$ and $(\rho \bb{u})^{n+1}(n\Delta t)=(\rho \bb{u})^{n}(n\Delta t)$ in weakly continuous sense in time;
    \item $\rho^{n+1} \in L^\infty(n\Delta t, (n+1)\Delta t; L^\gamma(\Omega^M))$, $\bb{u}^{n+1}\in L^2(n\Delta t, (n+1)\Delta t; H_0^1(\Omega^M))$, $\rho^{n+1}|\bb{u}^{n+1}|^2\in L^\infty(n\Delta t, (n+1)\Delta t; L^1(\Omega^M))$;
	\item The renormalized continuity equation
		\begin{align}
        &\int_{n\Delta t}^{(n+1)\Delta t} \int_{\Omega^M} \rho^{n+1} B(\rho^{n+1})( \partial_t \varphi +\bb{u}^{n+1}\cdot \nabla \varphi) \nonumber \\
         &= \int_{n\Delta t}^{(n+1)\Delta t} \int_{\Omega^M} b(\rho^{n+1})(\nabla\cdot \bb{u}^{n+1}) \varphi - \int_{\Omega^M} \rho^{n+1} B(\rho^{n+1}) \varphi \Big|_{n\Delta t}^{(n+1)\Delta t} \label{app:cont}
	   \end{align}
holds for all $\varphi \in C^\infty([n\Delta t,(n+1)\Delta t]\times \mathbb{R}^3)$ and any $b\in L^{\infty}(0,\infty) \cap C[0,\infty)$ such that $b(0)=0$ with $B(x)=B(1)+\int_1^x \frac{b(z)}{z^2}dz$;
\item The following momentum equation
	\begin{align}
        &\int_{n\Delta t}^{(n+1)\Delta t} \int_{\Omega^M} \rho^{n+1} \bb{u}^{n+1} \cdot\partial_t \boldsymbol\varphi + \int_{n\Delta t}^{(n+1)\Delta t} \int_{\Omega^M} (\rho^{n+1} \bb{u}^{n+1} \otimes \bb{u}^{n+1}):\nabla\boldsymbol\varphi\nonumber\\
        &  +\int_{n\Delta t}^{(n+1)\Delta t} \int_{\Omega^M} (\rho^{n+1})^\gamma (\nabla \cdot \boldsymbol\varphi)-  \int_{n\Delta t}^{(n+1)\Delta t} \int_{\Omega^M}   \mathbb{S}_{\eta,\varepsilon}^{n+1}(\nabla\bb{u}^{n+1}): \nabla \boldsymbol\varphi\nonumber\\
         & - \varepsilon\int_{n\Delta t}^{(n+1)\Delta t}  \int_\Gamma\ddfrac{\bb{v}^{n+1}-\partial_t \eta^{n+1}\bb{e}_z}{\Delta t} \boldsymbol\psi   = \int_{\Omega^M}\rho^{n+1}\bb{u}^{n+1}\cdot\boldsymbol\varphi \Big|_{n\Delta t}^{(n+1)\Delta t} \label{momeqFSP}
	\end{align}
holds for all $\boldsymbol\varphi \in C_c^\infty([n\Delta t,(n+1)\Delta t]\times \Omega^M)$, where $\boldsymbol\psi:=\gamma_{|\Gamma^{\eta^{n+1}}}\boldsymbol\varphi$ on $G_T$ and
    \begin{eqnarray*}
        &&\bb{v}^{n+1}=\gamma_{|\Gamma^{\eta^{n+1}}}\bb{u}^{n+1}, \quad \mathbb{S}_{\eta,\varepsilon}^{n+1}(\nabla \bb{u}):=\chi_{\eta^{n+1}}^\varepsilon\left(\mu \left( \nabla \bb{u} + \nabla^\tau \bb{u}-\frac{2}{3} \nabla \cdot \bb{u} \right) + \lambda \nabla \cdot \bb{u} \right),
	\end{eqnarray*}
with $\chi_{\eta^{n+1}}^\varepsilon$ being defined in $\eqref{chi}$; 
\item The following energy inequality holds for a.a. $t\in (n\Delta t, (n+1)\Delta t)$:
	\begin{eqnarray}
        &&\frac12\int_{\Omega^M} \rho^{n+1}|\bb{u}^{n+1}|^2(t) +  \int_{\Omega^M} \frac{(\rho^{n+1})^\gamma}{\gamma-1}(t)+\int_{n\Delta t}^t\int_{\Omega^M} \mathbb{S}_{\eta,\varepsilon}^{n+1} (\nabla\bb{u}^{n+1}):\nabla\bb{u}^{n+1}\nonumber\\
        &&\quad+\frac\varepsilon{2\Delta t} \int_{n\Delta t}^t \int_{\Gamma}|\bb{v}^{n+1}|^2+\frac\varepsilon{2\Delta t}  \int_{n\Delta t}^t \int_{\Gamma}|\bb{v}^{n+1}-\partial_t \eta^{n+1}\bb{e}_z|^2 \nonumber \\
        &&\leq \int_{n\Delta t}^{t}\int_{\Gamma} F \partial_t \eta^{n+1} +E_{fl}^n+\frac\varepsilon{2\Delta t} \int_{n\Delta t}^t\int_{\Gamma} |\partial_t \eta^{n+1}|^2. \nonumber \\ \label{en:ineq:fsp}
	\end{eqnarray}
\end{enumerate}
\subsection{Solving the approximate problems}
To solve (SSP), first fix $t_m$ as in $\eqref{ass:ssp}$
and set $\eta_m^{n+1}(n \Delta t) =\eta^{n}(t_m)$ and replace $\int_{\Gamma}\tilde{v}^n \psi(n\Delta t)$ with $\int_{\Gamma}\partial_t \eta^n(t_m) \psi(n\Delta t)$ in $\eqref{plateeqSSP}$. Now, the solution $\eta_m^{n+1}$ can be obtained as a limit of solutions to the same problem where $F_{con,\delta,m}^{n+1}$ is replaced with  $\frac1\kappa\chi_{\{\eta_m^{n+1}<\delta\}} (\partial_t\eta_m^{n+1})^-$ for $\kappa>0$. This penalizing force acts only when $\eta_m^{n+1}<\delta$ and when the velocity is negative. Thus, it has a positive sign (it pushes up) and it is fully dissipative, so the inequality $\eqref{en:ineq:ssp}$ can be easily obtained by simply multiplying the equation by $\partial_t \eta_m^{n+1}$ and integrating over $(n\Delta t,t)\times \Gamma$. By passing to the limit $\kappa\to 0$ and then $m\to\infty$, we obtain the solution to $\eqref{plateeqSSP}$ (see \cite[Section 3]{contactwave} for more details).

\bigskip

In order to solve (FSP), let $t_m$ be as in $\eqref{ass:fsp}$ and set $\rho_m^{n+1}(n\Delta t)= \rho^{n}(t_m)$ and $(\rho \bb{u})_m^{n+1}(n\Delta t)=(\rho \bb{u})^{n}(t_m)$. We can now use the standard theory for compressible Navier-Stokes, since the problem is defined on a fixed domain and the only difference is the penalization term which is a lower order term (the approximation of the viscous term does not affect the proof in any way since it does not violate parabolicity). The fluid velocity is spanned in a finite Galerkin basis with $k\in \mathbb{N}$ functions, the continuity equation is damped with $-a\Delta\rho$ for $a>0$, and an additional term $a \nabla\rho \cdot \nabla \bb{u}$ is added to the momentum equation to correct the energy. Finally, the pressure is regularized by adding a term $b\rho^\beta$ for $\beta\geq 4$. The solution to such problem is obtained by Schauder’s fixed point theorem. Then, passing to the limit $k\to \infty$, then $a\to 0$, $b\to 0$ and finally $m\to\infty$, the solution of (FSP) is obtained. For more details for the standard theory, see \cite[Chapter 3]{FeNobook} or \cite[Section 7]{forbiddenbook}.

\subsection{Coupled back problem and uniform estimates}
Denote 
\begin{eqnarray*}
    f(t):=f^{n}(t) \quad \text{ for } t\in [(n-1)\Delta t,n\Delta t).
\end{eqnarray*}
Fix $\psi \in C_c^\infty([0,T)\times \Gamma)$ and $\boldsymbol\varphi \in C_c^\infty([0,T)\times \Omega^M)$ such that $\gamma_{|\Gamma^{\eta}}\boldsymbol\varphi = \psi \bb{e}_z$. Now, in $\eqref{plateeqSSP}$ we test $\chi_{\{n\Delta t,t_m\}}\psi$ by the density argument, where $t_m\in (n\Delta t, (n+1)\Delta t)$ is as in $\eqref{ass:ssp}$, and then pass to the limit $m\to \infty$ to obtain
\begin{eqnarray}
    &&-(1-\varepsilon)\int_{n\Delta t}^{(n+1)\Delta t}\int_{\Gamma}\partial_t \eta^{n+1}\partial_t \psi   +\varepsilon\int_{n\Delta t}^{(n+1)\Delta t}\int_{\Gamma}\ddfrac{\partial_t \eta^{n+1} - T_{\Delta t} \bb{v}^{n}\cdot \bb{e}_z}{\Delta t}\psi \nonumber\\ 
     &&\quad +\varepsilon\int_{n\Delta t}^{(n+1)\Delta t}\int_{\Gamma} \partial_t \nabla \eta^{n+1}\cdot\nabla\psi +\int_{n\Delta t}^{(n+1)\Delta t}\int_{\Gamma}\Delta \eta^{n+1} \Delta \psi \nonumber\\ 
     &&= \int_{n\Delta t}^{(n+1)\Delta t}\int_{\Gamma} F_{con,\delta}^{n+1} - \int_{n\Delta t}^{(n+1)\Delta t}\int_{\Gamma} F^{n+1} \psi \nonumber\\
     &&\quad - (1-\varepsilon)\left(\int_\Gamma \tilde{v}^{n+1}\psi((n+1)\Delta t)-\int_\Gamma \tilde{v}^{n}\psi(n\Delta t)\right), \label{before:coupled}
\end{eqnarray}
for any $1\leq n\leq N-1$. Now, summing up $\eqref{before:coupled}$ and $\eqref{momeqFSP}$ and then summing up over $n=1,2,...,N-2$ and then summing up with $\eqref{plateeqSSP}$ and $\eqref{momeqFSP}$ for $n=N-1$ (and doing the same for the continuity equation $\eqref{app:cont}$), the approximate solutions $(\rho,\bb{u},\eta)$ satisfy:
	\begin{enumerate}
		\item The renormalized continuity equation
		\begin{eqnarray}
         \int_{Q_T^M} \rho B(\rho)( \partial_t \varphi +\bb{u}\cdot \nabla \varphi) =\int_{Q_T^M} b(\rho)(\nabla\cdot \bb{u}) \varphi -\int_{\Omega^M}\rho_{0,\delta}B(\rho_{0,\delta}) \varphi(0)   \label{coupledback:cont}
     \end{eqnarray}
holds for all functions $\varphi \in C_c^\infty([0,T)\times \mathbb{R}^3)$ and any $b\in L^\infty (0,\infty) \cap C[0,\infty)$ such that $b(0)=0$ with $B(\rho)$ being any primitive function to $\frac{b(\rho)}{\rho^2}$; 
\item The coupled momentum equation 
	\begin{eqnarray}
        &&\int_{Q_T^M} \rho \bb{u} \cdot\partial_t \boldsymbol\varphi + \int_{Q_T^M}(\rho \bb{u} \otimes \bb{u}):\nabla\boldsymbol\varphi +\int_{Q_T^M} \rho^\gamma (\nabla \cdot \boldsymbol\varphi)-  \int_{Q_T^M} \mathbb{S}_{\eta,\varepsilon}( \nabla\bb{u}): \nabla \boldsymbol\varphi \nonumber\\
        &&\quad+(1-\varepsilon) \int_{G_T} \partial_t \eta\partial_t \psi-\varepsilon\int_{G_T}\ddfrac{(\bb{v} - T_{\Delta t} \bb{v})\cdot \bb{e}_z}{\Delta t}\psi  - \int_{G_T}\Delta\eta \Delta\psi - \varepsilon\int_{G_T} \partial_t \nabla \eta \cdot \nabla \psi  \nonumber\\
        &&\quad+\int_{G_T} F_{con,\delta}\psi +\int_{G_T} F\psi \nonumber\\
        &&= -\int_{\Omega^M}(\rho\bb{u})_{0,\delta}\cdot\boldsymbol\varphi(0)- (1-\varepsilon)\int_\Gamma v_0 \psi(0)  \label{coupledback:mom}
    \end{eqnarray}
holds for all $\boldsymbol\varphi\in C_c^\infty([0,T)\times\Omega^M)$ and $\psi \in C_c^\infty([0,T)\times\Gamma)$ such that $\gamma_{|\Gamma^\eta}\boldsymbol\varphi = \psi \bb{e}_z$.
\end{enumerate}
In order to obtain the energy inequality, let $t_m\in (n\Delta t, (n+1)\Delta t)$ be as in $\eqref{ass:ssp}$. Evaluating $\eqref{en:ineq:ssp}$ at $t_m$ and passing to the limit $m\to \infty$ gives us
 \begin{eqnarray} 
    &&E_{str,\varepsilon}^{n+1}+\frac\varepsilon{2\Delta t}\int_{n\Delta t}^{(n+1)\Delta t}\int_{\Gamma}|\partial_t \eta - T_{\Delta t}\bb{v}\cdot \bb{e}_z|^2+ \frac\varepsilon{2\Delta t}\int_{n\Delta t}^{(n+1)\Delta t}\int_{\Gamma}|\partial_t \eta |^2
     \nonumber\\
      &&\leq \frac\varepsilon{2\Delta t}\int_{n\Delta t}^{(n+1)\Delta t}\int_{\Gamma} |T_{\Delta t}\bb{v}\cdot \bb{e}_z|^2 +E_{str,\varepsilon}^n,\label{coupledback:ssp:en}
	\end{eqnarray}
while choosing $t_m\in (n\Delta t, (n+1)\Delta t)$ as in $\eqref{ass:fsp}$ and passing to the limit $m\to \infty$ in $\eqref{en:ineq:fsp}$ gives us
    \begin{eqnarray}
		&& E_{fl}^{n+1} +\int_{n\Delta t}^{(n+1)\Delta t}\int_{\Omega^M} \mathbb{S}_{\eta,\varepsilon} (\nabla\bb{u}):\nabla\bb{u}\nonumber  \\
        &&\quad+\frac\varepsilon{2\Delta t} \int_{n\Delta t}^{(n+1)\Delta t}\int_{\Gamma}|\bb{v}|^2+\frac\varepsilon{2\Delta t}  \int_{n\Delta t}^{(n+1)\Delta t}\int_{\Gamma}|\bb{v}-\partial_t \eta\bb{e}_z|^2 \nonumber \\
		&&\leq \int_{n\Delta t}^{(n+1)\Delta t}\int_{\Gamma} F \partial_t \eta +E_{fl}^n+\frac\varepsilon{2\Delta t} \int_{n\Delta t}^{(n+1)\Delta t}\int_{\Gamma} |\partial_t \eta|^2(n\Delta t). \label{coupledback:fsp:en}
	\end{eqnarray}
Then, for any $m\in \{0,1,...,N-1\}$ and a.a. $t\in (m\Delta t, (m+1)\Delta t]$, we can sum up $\eqref{coupledback:ssp:en}$ and $\eqref{coupledback:fsp:en}$ and then sum over $n=0,1,...,m$. To the resulting inequality, we then add $\eqref{en:ineq:ssp}$ and $\eqref{en:ineq:fsp}$ evaluated at $t$, which in total by telescoping give us following energy inequality is satisfied:
	\begin{eqnarray}
        &&\frac12\int_{\Omega^M} \rho|\bb{u}|^2(t) +  \int_{\Omega^M} \frac{\rho^\gamma}{\gamma-1}(t)\nonumber\\
        &&\quad+\int_0^t\int_{\Omega^M} \mathbb{S}_{\eta,\varepsilon} (\nabla\bb{u}):\nabla\bb{u}+\frac\varepsilon{2\Delta t} \int_{m\Delta t}^t\int_{\Gamma}|\bb{v}|^2+\frac\varepsilon{2\Delta t}  \int_0^t\int_{\Gamma}|\bb{v}-\partial_t \eta \bb{e}_z|^2  \nonumber \\
        &&\quad+\frac\varepsilon{2\Delta t}\int_0^t\int_{\Gamma}|\partial_t \eta - T_{\Delta t}\bb{v}\cdot \bb{e}_z|^2  + \frac{1-\varepsilon}2\int_{\Gamma}|\partial_t \eta|^2(t)+\frac12\int_{\Gamma}|\Delta\eta|^2(t) + \varepsilon \int_0^t \int_\Gamma |\partial_t \nabla \eta|^2 \nonumber\\
        &&\leq  \int_0^t\int_{G_T} F\partial_t \eta+ \displaystyle{\int_{\Omega^M} \left( \frac{1}{2\rho_{0,\delta}} |(\rho\bb{u})_{0,\delta}|^2 + \frac{\rho_{0,\delta}^\gamma}{\gamma-1}   \right)} + \int_{\Gamma} \left(\frac{1-\varepsilon}{2}|v_0|^2 + \frac{1}{2} |\Delta \eta_{0,\delta}|^2  \right) .\nonumber \\
         && \label{coupled:en:app}
	\end{eqnarray}
In order to extract uniform estimates from the above inequality, we need to control the following term
	\begin{eqnarray*}
		\int_0^t \int_\Gamma F \partial_t \eta \leq \|F\|_{L^1(0,T; L^2(\Gamma))} \|\partial_t \eta \|_{L^\infty(0,T; L^2(\Gamma))} \leq 2\|F\|_{L^1(0,T; L^2(\Gamma))}^2 + \frac18 \|\partial_t \eta \|_{L^\infty(0,T; L^2(\Gamma))}^2, 
	\end{eqnarray*}
so taking supremum over $t$ in $\eqref{coupled:en:app}$ we obtain that $\|\partial_t \eta \|_{L^\infty(0,T; L^2(\Gamma))} \leq C$ and consequently the right-hand side of $\eqref{coupled:en:app}$ is now uniformly bounded. Therefore, we have the following uniform estimates:
	\begin{eqnarray}
		&&\esssup_{t\in (0,T)}  \left(\int_{\Omega^M} \rho|\bb{u}|^2(t) +  \int_{\Omega^M} \rho^\gamma(t) \right) \leq C, \label{uni1}\\
		&&\int_{Q_T^M} \mathbb{S}_{\eta,\varepsilon} (\nabla\bb{u}):\nabla\bb{u} + \varepsilon \int_{G_T} |\partial_t \nabla \eta|^2\leq C, \label{uni2}\\
		&&\esssup_{t\in (0,T)}  \left( \int_{\Gamma}|\partial_t \eta|^2(t)+
		\int_{\Gamma}|\Delta\eta|^2(t) \right) \leq C,\label{uni3} \\
		&& \int_{G_T}|\bb{v}-\partial_t \eta \bb{e}_z|^2 \leq C\Delta t. \label{kin:coup:dt}
	\end{eqnarray}
This in particular implies
	\begin{eqnarray}
		\|\bb{u}\|_{L^2(0,T; L^6(\Omega^M))}\leq C\|\bb{u}\|_{L^2(0,T; H^1(\Omega^M))} \leq C(\varepsilon). \label{bbu:eps}
	\end{eqnarray}
In order to estimate the maximum of $\eta$, let us first note that
	\begin{eqnarray*}
		|\eta(t,x)|^2 \leq 2|\eta_{0,\delta}(x)|^2 + 2\left|\int_0^t \partial_t \eta(\tau,x)d\tau\right|^2 \leq 2|\eta_{0,\delta}(x)|^2 + 2t\int_0^t |\partial_t \eta(\tau,x)|^2d\tau,
	\end{eqnarray*}
so integrating over $\Gamma$ and taking the supremum over $t$ gives us
	\begin{eqnarray*}
		&&\|\eta\|_{L^\infty(0,T; L^2(\Gamma))}^2 \leq 2 \|\eta_{0,\delta}\|_{L^2(\Gamma)}^2 + 2T \|\partial_t \eta\|_{L^2(0,T; L^2(\Gamma))}^2\\
		&&\leq 2 \|\eta_{0,\delta}\|_{L^2(\Gamma)}^2 + 2T^2 \|\partial_t \eta\|_{L^\infty(0,T; L^2(\Gamma))}^2 \\
		&&\leq   C(1+T^2).
	\end{eqnarray*}
Since $\eta$ is space-periodic, one has
	\begin{eqnarray}
		\|\eta\|_{L^\infty((0,T)\times \Gamma)} \leq C(\|\eta\|_{L^\infty(0,T; L^2(\Gamma))}+\|\Delta\eta\|_{L^\infty(0,T; L^2(\Gamma))})\leq C(1+T). \label{eta:max:est}
	\end{eqnarray}
Therefore, the constant $M$ which is the height of the extended domain can a priori be chosen for a fixed $T>0$ so that, say, $M/2\geq \|\eta\|_{L^\infty((0,T)\times \Gamma)}$.

\subsection{Passing to the limit $\Delta t\to 0$}
Let us denote the solution obtained in the previous section as $(\rho_{\Delta t},\bb{u}_{\Delta t},\eta_{\Delta t})$ which satisfy $\eqref{coupledback:cont},\eqref{coupledback:mom},\eqref{coupled:en:app}$ and denote the following weak limits of a converging subsequence as $\Delta t\to 0$:
\begin{eqnarray*}
    &\rho_{\Delta t} \rightharpoonup \rho &\quad \text{ weakly* in }\quad L^\infty(0,T; L^\gamma(\Omega^M)), \\
    &\bb{u}_{\Delta t} \rightharpoonup \bb{u}&\quad \text{ weakly in }\quad L^2(0,T; H_0^1(\Omega^M)), \\
    &\eta_{\Delta t} \rightharpoonup \eta& \quad \text{ weakly* in }\quad W^{1,\infty}(0,T; L^2(\Gamma))\cap L^\infty(0,T; H^2(\Gamma)).
\end{eqnarray*}
First, let us point out that the last convergence implies
    \begin{eqnarray}
         \eta_{\Delta t} \to \eta, \quad \text{ in } C^{0,\alpha}([0,T]\times \Gamma), \label{eta:deltat}
    \end{eqnarray}
for any $\alpha\in (0,1/3)$, by \cite[Theorem 5.2]{AmannGlasnik}. For all means and purposes, this implies a good enough convergence of the fluid domain throughout all the limits. \\
    
Now, due to estimate $\eqref{kin:coup:dt}$, we obtain that $\partial_t \eta\bb{e}_z = \bb{v}=\gamma_{|\Gamma^\eta}\bb{u}$, so the kinematic coupling is satisfied and consequently
	\begin{eqnarray*}
		&&(1-\varepsilon)\int_{G_T} \partial_t\eta_{\Delta t}\partial_t \psi -\varepsilon\int_{G_T}\frac{(\bb{v}_{\Delta t} - T_{\Delta t} \bb{v}_{\Delta t})\cdot \bb{e}_z}{\Delta t}\psi+ (1-\varepsilon)\int_\Gamma v_0 \psi\\
         &&\to \int_{G_T} \partial_t\eta\partial_t \psi +\int_\Gamma v_0 \psi
	\end{eqnarray*}
for every $\psi \in C_c^\infty([0,T)\times\Gamma)$. \\

Next, the convergence of all fluid terms except the pressure follow by the standard arguments (see \cite[Section 3.6]{FeNobook}) which are localized away from the plate $\Gamma^\eta$ as follows. First, noting that
\begin{eqnarray*}
    &&\rho_{\Delta t} \bb{u}_{\Delta t}\in L^\infty(0,T; L^{\frac{2\gamma}{\gamma+1}}(\Omega^M)), \quad \rho_{\Delta t} \bb{u}_{\Delta t}\otimes \bb{u}_{\Delta t} \in L^2(0,T;L^{\frac{12\gamma}{6+2\gamma}}(\Omega^M) ),\\
    && \mathbb{S}_{\eta_{\Delta t},\varepsilon}(\nabla \bb{u}_{\Delta t}) \in L^2(Q_T^M)
\end{eqnarray*}
by $\eqref{uni1},\eqref{uni2},\eqref{bbu:eps}$, it is enough to identify these terms locally. Fix an open set $Q:=I\times O\subset Q_T^M$ such that $I\subset (0,T)$ and $O\subset \Omega^M\setminus \Gamma^\eta(t)$ for every $t\in I$. Then, by $\eqref{eta:deltat}$, there exists $\Delta t_Q$ such that for every $\Delta t<\Delta t_Q$ one has $O\subset \Omega^M\setminus \Gamma^{\eta_{\Delta t}}(t)$ for every $t\in I$. Now, from the continuity equation $\eqref{coupledback:cont}$, one has
\begin{eqnarray*}
    \partial_t \rho_{\Delta t} = - \nabla\cdot (\rho_{\Delta t} \bb{u}_{\Delta t}) \in L^2(I; H^{-1}(O))
\end{eqnarray*}
so $\rho_{\Delta t} \in L^\infty(I; L^\gamma(O))\cap H^1(I ; H^{-1}(O))\hookrightarrow\hookrightarrow L^2(I; H^{-1}(O))$, and since $\bb{u}_{\Delta t} \rightharpoonup \bb{u}$ weakly in $L^2(I; H^1(O))$, we conclude 
\begin{eqnarray*}
    \rho_{\Delta t} \bb{u}_{\Delta t} \rightharpoonup \rho\bb{u} \quad \text{ weakly in } L^2(I; L^{2}(O)).
\end{eqnarray*}
Next, from the momentum equation $\eqref{coupledback:mom}$, we have $\rho_{\Delta t} \bb{u}_{\Delta t} \in H^1(I; H^{-3}(O))$, while $\rho_{\Delta t} \bb{u}_{\Delta t}\in L^2(I; L^{2}(O))$ so $\rho_{\Delta t}\bb{u}_{\Delta t} \to \rho\bb{u}$ in $L^2(I; H^{-1}(O))$ by Aubin-Lions lemma. Once again, since $\bb{u}_{\Delta t} \rightharpoonup \bb{u}$ weakly in $L^2(I; H^1(O))$, one obtains
\begin{eqnarray}
    \rho_{\Delta t} \bb{u}_{\Delta t}\otimes \bb{u}_{\Delta t} \rightharpoonup \rho \bb{u}\otimes \bb{u}  \quad \text{ weakly in } L^2(I;L^{\frac{12\gamma}{6+2\gamma}}(O) ). \label{convrhouu}
\end{eqnarray}
Since the viscous term is linear in $\bb{u}_{\Delta t}$, its convergence is trivial so we therefore summarize the above arguments to conclude
\begin{eqnarray*}
    &&\lim_{\Delta t\to 0}\int_{Q_T^M}\rho_{\Delta t} \bb{u}_{\Delta t} \cdot \partial_t\boldsymbol\varphi =\int_{Q_T^M}\rho \bb{u} \cdot \partial_t\boldsymbol\varphi, ~  \lim_{\Delta t\to 0}\int_{Q_T^M}\rho_{\Delta t} \bb{u}_{\Delta t}\otimes \bb{u}_{\Delta t}:\nabla \boldsymbol{\varphi} =  \int_{Q_T^M} \rho \bb{u} \otimes \bb{u}:\nabla \boldsymbol{\varphi}, \\
    &&\lim_{\Delta t\to 0}\int_{Q_T^M}\mathbb{S}_{\eta_{\Delta t},\varepsilon}(\nabla\bb{u}_{\Delta t}):\nabla\boldsymbol\varphi =\int_{Q_T^M}\mathbb{S}_{\eta,\varepsilon}(\nabla\bb{u}):\nabla\boldsymbol\varphi,
\end{eqnarray*}
for every $\boldsymbol\varphi \in C^\infty([0,T]\times\mathbb{R}^3)$. Next, since the pressure $\rho_{\Delta t}^\gamma$ is only bounded in $L^1$ in space, we need to exclude the possibility of it converging to a measure. The following proof can be found in literature, however we present it here since it will be referred later in Lemma $\ref{weak:conv:pr:delta}$ when we show weak convergence of pressure even with contact.
\begin{lem}\label{weak:conv:pr}
There exists $\overline{p}\in L^1(Q_T^\eta)$ such that 
    \begin{eqnarray*}
        \lim_{\Delta t\to 0}\int_{Q_T^{\eta_{\Delta t}}}\rho_{\Delta t}^\gamma \varphi \to \int_{Q_T^\eta}\overline{p} \varphi,
    \end{eqnarray*}
for any $\varphi\in C^\infty([0,T]\times \mathbb{R}^3)$.
\end{lem}
\begin{proof}
The improved pressure estimates which in the standard framework on fixed domains hold up to boundary (see \cite[Section 2.2.5]{FeNobook}), in our scenario hold away from the $\Gamma^\eta$ and need to be localized. 

First, by $\eqref{eta:deltat}$, for every open set $Q=I\times O \subset Q_T^\eta$, there is a $\Delta t_Q$ such that for every $\Delta t<\Delta t_Q$ we have $Q\subset Q_T^{\eta_{\Delta t}}$. Fix such $Q$ and let $I'\times O'=:Q'\supset Q$ be an open set with the same properties so that both $Q,Q'$ are sufficiently regular. By testing the momentum equation $\eqref{coupledback:mom}$ with $\psi \nabla \Delta_{O'}^{-1}\rho_{\Delta t}$, where $\Delta_{O'}^{-1}$ is the inverse Laplacian on $O'$ with zero Dirichlet boundary data and $\psi \in C_c^\infty(Q')$ is non-negative and $\psi=1$ on $Q$ (see for example in \cite[Lemma 6.3]{Breit}), leads to
\begin{eqnarray*}
    \rho_{\Delta t} \in L^{\gamma+1}(Q),
\end{eqnarray*}
so $\rho_{\Delta t}^\gamma \rightharpoonup \overline{p}$ weakly in $L^{\frac{\gamma+1}{\gamma}}(Q)$. 
Therefore, there exists a function $\overline{p}\in L^1(Q_T^\eta)$ defined locally as a weak limit of $\rho_{\Delta t}^\gamma$. However, we still need to ensure there are no concentrations at $\partial\Omega^\eta$. This is done by adapting the arguments from \cite{kukucka} to our fluid-structure interaction setting, as in \cite[Lemma 6.4]{Breit}. Fix $0<\kappa<\delta/4$ and let
\begin{eqnarray*}
    \boldsymbol\varphi_{\Delta t}^\kappa(t,x,y,z):= 
    \begin{cases}
         \frac{z-\eta_{\Delta t}(t,x,y)}{\kappa},& \text{ for } \eta_{\Delta t}(t,x,y)-\kappa< z < \eta_{\Delta t}(t,x,y),\\
        \frac{\eta_{\Delta t}(t,x,y)-2 z}{\eta_{\Delta t}(t,x,y) - 2\kappa},& \text{ for } \kappa\leq z\leq \eta_{\Delta t}(t,x,y)-\kappa,\\
        \frac{z}\kappa , & \text{ for } 0<z<\kappa ,\\
        0, & \text{ elsewhere on } Q_T^M.
    \end{cases}
\end{eqnarray*}
Testing the momentum equation with $(\varphi_{\Delta t}^\kappa \bb{e}_z, 0)$, one has
    \begin{eqnarray*}
        &&\int_0^T \int_\Gamma\left(\int_0^\kappa+\int_{\eta_{\Delta t}-\kappa}^{ \eta_{\Delta t}}\right)  \frac1\kappa\rho_{\Delta t}^\gamma \nonumber\\
        &&=\int_{Q_T^{\eta_{\Delta t}}} \rho_{\Delta t}^\gamma (\nabla \cdot \boldsymbol\varphi_{\Delta t}^\kappa)- \int_0^T\int_\Gamma\int_\kappa^{\eta_{\Delta t}-\kappa} \rho_{\Delta t}^\gamma (\nabla \cdot \boldsymbol\varphi_{\Delta t}^\kappa)\nonumber\\
        &&=-\int_{Q_T^{\eta_{\Delta t}}} \rho_{\Delta t} \bb{u}_{\Delta t} \cdot\partial_t \boldsymbol\varphi_{\Delta t}^\kappa - \int_{Q_T^{\eta_{\Delta t}}}(\rho_{\Delta t} \bb{u}_{\Delta t} \otimes \bb{u}_{\Delta t}):\nabla\boldsymbol\varphi_{\Delta t}^\kappa
         +  \int_{Q_T^{\eta_{\Delta t}}} \mathbb{S}_{\eta_{\Delta t},\varepsilon}( \nabla\bb{u}_{\Delta t}): \nabla \boldsymbol\varphi_{\Delta t}^\kappa \nonumber\\
        &&\quad + \int_{\Omega^{\eta_{\Delta t}}(T)}\rho_{\Delta t}(T)\bb{u}_{\Delta t}(T)\cdot \boldsymbol\varphi_{\Delta t}^\kappa(T)-\int_{\Omega^{\eta_{0,\delta}}}(\rho \bb{u})_{0,\delta}\cdot\boldsymbol\varphi_{\Delta t}^\kappa(0)-\int_0^T\int_\Gamma\int_\kappa^{\eta_{\Delta t}-\kappa} \rho_{\Delta t}^\gamma (\nabla \cdot \boldsymbol\varphi_{\Delta t}^\kappa).
	\end{eqnarray*}
Let us bound the right-hand side, starting with
\begin{eqnarray*}
    &&\left|\int_{Q_T^{\eta_{\Delta t}}}\rho_{\Delta t} \bb{u}_{\Delta t} \cdot\partial_t \boldsymbol\varphi_{\Delta t}^\kappa\right| \\
    &&\leq \|\rho_{\Delta t}\|_{L^\infty(0,T; L^\gamma(\Omega^{\eta_{\Delta t}}))} \|\bb{u}_{\Delta t} \|_{L^2(0,T; L^6(\Omega^{\eta_{\Delta t}}))} \\
    &&\times\Bigg(\underbrace{\left|\left|\int_\Gamma\left(\int_0^\kappa+\int_{\eta_{\Delta t}-\kappa}^{ \eta_{\Delta t}} \right) |\partial_t  \boldsymbol\varphi_{\Delta t}^\kappa|^2  \right|\right|_{L^\infty(0,T)}^{\frac12}}_{\leq \frac{C}\kappa} \underbrace{\left|\left|\int_\Gamma\left(\int_0^\kappa+\int_{\eta_{\Delta t}-\kappa}^{ \eta_{\Delta t}}   \right) 1^r  \right|\right|_{L^\infty(0,T)}^{\frac1r}}_{\leq \kappa^\frac1r} \\
    && \qquad +   \underbrace{\left|\left|\int_\Gamma\int_\kappa^{\eta_{\Delta t}-\kappa} |\partial_t  \boldsymbol\varphi_{\Delta t}^\kappa|^2  \right|\right|_{L^\infty(0,T)}^{\frac12} \left|\left|\int_\Gamma\int_\kappa^{\eta_{\Delta t}-\kappa}   1^r  \right|\right|_{L^\infty(0,T)}^{\frac1r} }_{\leq C(\delta)}              \Bigg) \\    
    &&\leq C\left(\frac{1}{\kappa^{1-\frac1r}} +1 \right),
\end{eqnarray*}
where $r = (1 - \frac1\gamma-\frac16 - \frac12)^{-1}$. Next, one has
\begin{eqnarray*}
   &&\left| \int_{Q_T^{\eta_{\Delta t}}}(\rho_{\Delta t} \bb{u}_{\Delta t} \otimes \bb{u}_{\Delta t}):\nabla\boldsymbol\varphi_{\Delta t}^\kappa \right|\\
   &&\leq \|\rho_{\Delta t}\|_{L^\infty(0,T; L^\gamma(\Omega^{\eta_{\Delta t}}))} \|\bb{u}_{\Delta t} \|_{L^2(0,T; L^6(\Omega^{\eta_{\Delta t}}))}^2  \nonumber\\
   &&\times\Bigg(\underbrace{\left|\left|\int_\Gamma\left(\int_0^\kappa+\int_{\eta_{\Delta t}-\kappa}^{ \eta_{\Delta t}} \right) |\nabla  \boldsymbol\varphi_{\Delta t}^\kappa|^q  \right|\right|_{L^\infty(0,T)}^{\frac1q}}_{\leq \frac{C}\kappa} \underbrace{\left|\left|\int_\Gamma\left(\int_0^\kappa+\int_{\eta_{\Delta t}-\kappa}^{ \eta_{\Delta t}}   \right) 1^r  \right|\right|_{L^\infty(0,T)}^{\frac1r}}_{\leq \kappa^{\frac1r}} \\
    && \qquad +   \underbrace{\left|\left|\int_\Gamma\int_\kappa^{\eta_{\Delta t}-\kappa} |\nabla  \boldsymbol\varphi_{\Delta t}^\kappa|^q  \right|\right|_{L^\infty(0,T)}^{\frac1q} \left|\left|\int_\Gamma\int_\kappa^{\eta_{\Delta t}-\kappa}   1^r  \right|\right|_{L^\infty(0,T)}^{\frac1r} }_{\leq C(\delta)}              \Bigg) \\
   &&\leq C\left(\frac{1}{\kappa^{1-\frac1r}} +1 \right),
\end{eqnarray*}
where $q<\infty$ and $1<r = (1 - \frac1\gamma-\frac26 - \frac1q)^{-1}$. As the other terms can be bounded in a simpler fashion, we arrive at
\begin{eqnarray*}
   &&\int_0^T \int_\Gamma\left(\int_0^\kappa+\int_{\eta_{\Delta t}-\kappa}^{ \eta_{\Delta t}}\right) \frac1\kappa\rho_{\Delta t}^\gamma \leq C(\kappa^{-s}+1), \quad  \text{for some } s<1,
\end{eqnarray*}
where $C$ does not depend on $\Delta t$, so
\begin{eqnarray*}
     &&\int_0^T \int_\Gamma\left(\int_0^\kappa+\int_{\eta_{\Delta t}-\kappa}^{ \eta_{\Delta t}}\right)  \frac1\kappa\rho_{\Delta t}^\gamma  \leq C(\kappa^{-s+1}+\kappa).
\end{eqnarray*}
Therefore,
\begin{eqnarray*}
    &&\lim_{\Delta t\to 0}\int_{Q_T^{\eta_{\Delta t}}} \rho_{\Delta t}^\gamma \varphi  \\
    &&=\lim_{\Delta t\to 0}\int_0^T \int_\Gamma \int_\kappa^{\eta_{\Delta t}-\kappa} \rho_{\Delta t}^\gamma \varphi  +\lim_{\Delta t\to 0}\int_0^T  \int_\Gamma\left(\int_0^\kappa+\int_{\eta_{\Delta t}-\kappa}^{ \eta_{\Delta t}}\right) \rho_{\Delta t}^\gamma \varphi \\
    && = \int_0^T \int_\Gamma \int_\eta^{\eta -\kappa} \overline{p} \varphi  + o(\kappa^{s'}),
\end{eqnarray*}
for some $s'>0$. Passing to the limit $\kappa\to 0$, the proof is complete.

\end{proof}
As the fluid domain at this approximation level still includes the region above $\Gamma^\eta$, we also need the following weak convergence which is proved in exact same way:
\begin{cor}
There exists $\overline{p}_M\in L^1(Q_T^M\setminus Q_T^\eta)$ such that 
    \begin{eqnarray*}
         \lim_{\Delta t\to 0}\int_0^T\int_{\Omega^M\setminus\Omega^{\eta_{\Delta t}} }\rho_{\Delta t}^\gamma \varphi = \int_0^T\int_{\Omega^M\setminus\Omega^\eta }\overline{p}_M \varphi
    \end{eqnarray*}
for any $\varphi\in C^\infty([0,T]\times \mathbb{R}^3)$.
\end{cor}

Now, in order to identify the pressure, the approach developed by Lions in \cite{Lions} based on convergence of effective viscous flux and the renormalized continuity equation can be done locally (see also \cite[Section 6]{Breit}) to conclude
\begin{eqnarray*}
    \rho_{\Delta t} \to  \rho \quad  \text{ a.e. on }Q_T^\eta,
\end{eqnarray*} 
which by the weak convergence of pressure and Vitali convergence theorem implies
\begin{eqnarray*}
    \lim_{\Delta t\to 0}\int_{Q_T^M}\rho_{\Delta t}^\gamma \nabla\cdot \boldsymbol\varphi = \int_{Q_T^M}\rho^\gamma \nabla\cdot \boldsymbol\varphi 
\end{eqnarray*}
for every $C^\infty([0,T]\times\mathbb{R}^3)$. Finally, kinematic coupling combined with the fact that $\rho_{0,\delta}=0$ a.e. on $\Omega^M\setminus\Omega^{\eta_{0,\delta}}$ implies that
    \begin{eqnarray*}
    \rho = 0 \quad \text{ on } Q_T^M\setminus Q_T^\eta
\end{eqnarray*}
by \cite[Lemma 3.1]{heatexchange}, as no mass can pass through the elastic boundary anymore. The proof of the mentioned lemma follows the approach from \cite[Lemma 4.1]{commoving}, which considers compressible fluids on moving domains. Thus, $(\rho,\bb{u},\eta)$ satisfy: \\
\begin{enumerate}
\item The renormalized continuity equation
    \begin{eqnarray}
        \int_{Q_T^\eta} \rho B(\rho)( \partial_t \varphi +\bb{u}\cdot \nabla \varphi) =\int_{Q_T^\eta} b(\rho)(\nabla\cdot \bb{u}) \varphi -\int_{\Omega^{\eta_{0,\delta}}}\rho_{0,\delta}B(\rho_{0,\delta}) \varphi(0)   \label{coupledback:cont:eps}
	\end{eqnarray}
holds for all functions $\varphi \in C_c^\infty([0,T)\times \mathbb{R}^3)$ and any $b\in L^\infty (0,\infty) \cap C[0,\infty)$ such that $b(0)=0$ with $B(\rho)$ being any primitive function to $\frac{b(\rho)}{\rho^2}$; 
\item The coupled momentum equation 
    \begin{eqnarray}
         &&\int_{Q_T^\eta} \rho \bb{u} \cdot\partial_t \boldsymbol\varphi + \int_{Q_T^\eta}(\rho \bb{u} \otimes \bb{u}):\nabla\boldsymbol\varphi +\int_{Q_T^\eta} \rho^\gamma (\nabla \cdot \boldsymbol\varphi)-  \int_{Q_T^M} \mathbb{S}_{\eta,\varepsilon}( \nabla\bb{u}): \nabla \boldsymbol\varphi \nonumber\\
         &&+ \int_{G_T} \partial_t \eta\partial_t \psi - \int_{G_T}\Delta\eta \Delta\psi - \varepsilon\int_{G_T} \partial_t \nabla \eta \cdot \nabla \psi +\int_{G_T} F_{con,\delta}\psi +\int_{G_T} F\psi \nonumber\\
        &&= -\int_{\Omega^{\eta_{0,\delta}}}(\rho\bb{u})_{0,\delta}\cdot\boldsymbol\varphi(0)- \int_\Gamma v_0 \psi(0)  \label{coupledback:mom:eps}
	\end{eqnarray}
holds for all $\boldsymbol\varphi\in C_c^\infty([0,T)\times\mathbb{R}^3)$ and $\psi \in C_c^\infty([0,T)\times\Gamma)$ such that $\gamma_{|\Gamma^\eta}\boldsymbol\varphi = \psi \bb{e}_z$ and $\boldsymbol{\varphi}=0$ on $\{z=0\}$;
 \item The following energy inequality
 \begin{eqnarray}
     &&\frac12\int_{\Omega^\eta(t)} \rho|\bb{u}|^2(t) +  \int_{\Omega^\eta(t)} \frac{\rho^\gamma}{\gamma-1}(t)+\int_0^t\int_{\Omega^M} \mathbb{S}_{\eta,\varepsilon} (\nabla\bb{u}):\nabla\bb{u}+  \nonumber \\
     &&\quad + \frac{1-\varepsilon}2\int_{\Gamma}|\partial_t \eta|^2(t)+\frac12\int_{\Gamma}|\Delta\eta|^2(t) + \varepsilon \int_0^t \int_\Gamma |\partial_t \nabla \eta|^2 \nonumber\\
    &&\leq  \int_0^t\int_{G_T} F\partial_t \eta+ \displaystyle{\int_{\Omega^{\eta_{0,\delta}}} \left( \frac{1}{2\rho_{0,\delta}} |(\rho\bb{u})_{0,\delta}|^2 + \frac{\rho_{0,\delta}^\gamma}{\gamma-1}   \right)} + \int_{\Gamma} \left(\frac{1-\varepsilon}{2}|v_0|^2 + \frac{1}{2} |\Delta \eta_{0,\delta}|^2  \right) .\nonumber \\
     && \label{coupled:en:app:eps}
 \end{eqnarray}
 holds for a.a. $t\in (0,T)$.
\end{enumerate}

\subsection{Passing to the limit $\varepsilon \to 0$}
Denote the solution obtained in the previous section as $(\rho_{\varepsilon},\bb{u}_{\varepsilon },\eta_{\varepsilon})$ which satisfies $\eqref{coupledback:cont:eps},\eqref{coupledback:mom:eps},\eqref{coupled:en:app:eps}$ and denote the corresponding weak limits of a converging subsequence as $(\rho,\bb{u},\eta)$ as $\varepsilon \to 0$, where the convergence on the moving domain is to be understood as in \cite[
Definition 2.7]{Breit}, or more precisely:
\begin{eqnarray*}
    &\rho_\varepsilon\chi_{|\Omega^{\eta_\varepsilon}} \rightharpoonup \rho \chi_{|\Omega^{\eta}}&\quad \text{ weakly* in }\quad L^\infty(0,T; L^\gamma( \mathbb{R}^3)), \\
    &\bb{u}_\varepsilon\chi_{|\Omega^{\eta_\varepsilon}} \rightharpoonup \bb{u}\chi_{|\Omega^{\eta}}&\quad \text{ weakly in }\quad L^2((0,T)\times \mathbb{R}^3), \\
    &\nabla\bb{u}_\varepsilon\chi_{|\Omega^{\eta_\varepsilon}} \rightharpoonup \nabla\bb{u}\chi_{|\Omega^{\eta}}&\quad \text{ weakly in }\quad L^2((0,T)\times \mathbb{R}^3),
\end{eqnarray*}
where by construction $\bb{u}\in L^2(0,T; H_{0z}^1(\Omega^\eta(t)))$, while once again
\begin{eqnarray*}
     &\eta_\varepsilon \rightharpoonup \eta& \quad \text{ weakly* in }\quad W^{1,\infty}(0,T; L^2(\Gamma))\cap L^\infty(0,T; H^2(\Gamma)).
\end{eqnarray*}
Fix $\boldsymbol\varphi\in C^\infty([0,T]\times\mathbb{R}^3)$. First, it is worth noting that
	\begin{eqnarray*}
		&&\int_{Q_T^M} \mathbb{S}_{\eta_\varepsilon,\varepsilon}( \nabla\bb{u}_\varepsilon): \nabla \boldsymbol\varphi = \int_{Q_T^{\eta_\varepsilon}} \mathbb{S}_{\eta_\varepsilon,\varepsilon}( \nabla\bb{u}_\varepsilon): \nabla \boldsymbol\varphi + \int_{Q_T^M\setminus Q_T^{\eta_\varepsilon}} \mathbb{S}_{\eta_\varepsilon,\varepsilon}( \nabla\bb{u}_\varepsilon): \nabla \boldsymbol\varphi ,
	\end{eqnarray*}
while
	\begin{eqnarray*}
		\int_{Q_T^M\setminus Q_T^{\eta_\varepsilon}} \mathbb{S}_{\eta_\varepsilon,\varepsilon}( \nabla\bb{u}_\varepsilon): \nabla \boldsymbol\varphi &=& \int_0^T \int_\Gamma \int_{\eta_\varepsilon+\varepsilon^2}^M\underbrace{\mathbb{S}_{\eta_\varepsilon,\varepsilon}( \nabla\bb{u}_\varepsilon)}_{=\varepsilon \mathbb{S}( \nabla\bb{u}_\varepsilon)}: \nabla \boldsymbol\varphi \\
         &&+ \int_0^T \int_\Gamma \int_{\eta_\varepsilon}^{\eta_\varepsilon+\varepsilon^2}\mathbb{S}_{\eta_\varepsilon,\varepsilon}( \nabla\bb{u}_\varepsilon): \nabla \boldsymbol\varphi \to 0 
	\end{eqnarray*}
as $\varepsilon\to 0$, since 
    \begin{eqnarray*}
        && \int_0^T \int_\Gamma \int_{\eta_\varepsilon}^{\eta_\varepsilon+\varepsilon^2}\mathbb{S}_{\eta_\varepsilon,\varepsilon}( \nabla\bb{u}_\varepsilon): \nabla \boldsymbol\varphi \leq C \Bigg(\underbrace{\int_0^T \int_\Gamma \int_{\eta_\varepsilon}^{\eta_\varepsilon+\varepsilon^2}|\nabla\bb{u}_\varepsilon|^2 }_{\leq \frac{C}{\varepsilon} \text{ by } \eqref{chi},\eqref{coupled:en:app:eps}}\Bigg)^{\frac12} \left(\int_0^T \int_\Gamma \int_{\eta_\varepsilon}^{\eta_\varepsilon+\varepsilon^2} 1^2\right)^{\frac12} \leq C\sqrt{\varepsilon},\\        
        &&\sqrt{\varepsilon}\int_0^T \int_\Gamma \int_{\eta_\varepsilon+\varepsilon^2}^M\mathbb{S}(\nabla\bb{u}_\varepsilon):\nabla\bb{u} \leq C,
    \end{eqnarray*}
so we conclude
	\begin{eqnarray*}
		\lim_{\varepsilon\to 0}\int_{Q_T^M} \mathbb{S}_{\eta_\varepsilon,\varepsilon}( \nabla\bb{u}_\varepsilon): \nabla \boldsymbol\varphi \to \int_{Q_T^\eta} \mathbb{S}( \nabla\bb{u}): \nabla \boldsymbol\varphi.
	\end{eqnarray*}
Next, by Lemma $\ref{H1}$ and $\eqref{eta:max:est}$
	\begin{eqnarray*}
		\|\bb{u}\|_{L^2(0,T; H^1(\Omega^\eta))}^2 \leq (1+\|\eta\|_{L^\infty((0,T)\times\Gamma)}^2)\int_{Q_T^\eta} \mathbb{S}(\nabla \bb{u}):\nabla \bb{u}\leq C(1+T^2),
	\end{eqnarray*}
and consequently
	\begin{eqnarray*}
		\|\bb{u}\|_{L^2(0,T;L^p(\Omega^\eta))} \leq C \|\bb{u}\|_{L^2(0,T; H^1(\Omega^\eta))} \leq C(1+T^2),
	\end{eqnarray*}
for any $p<6$ by Lemma $\ref{imbed}$.\\
	
The compactness arguments stay the same as in the limit $\Delta t \to 0$. The limiting functions $(\rho,\bb{u},\eta)$ satisfy the renormalized continuity equation $\eqref{reconteqweak}$, the energy inequality $\eqref{en:ineq}$ and the coupled momentum equation with the presence of contact force, or more precisely:
    \begin{eqnarray}
		&&\int_{Q_T^\eta} \rho \bb{u} \cdot\partial_t \boldsymbol\varphi + \int_{Q_T^\eta}(\rho \bb{u} \otimes \bb{u}):\nabla\boldsymbol\varphi +\int_{Q_T^\eta} \rho^\gamma (\nabla \cdot \boldsymbol\varphi)-  \int_{Q_T^\eta} \mathbb{S}( \nabla\bb{u}): \nabla \boldsymbol\varphi \nonumber\\
		&&\quad +\int_{G_T} \partial_t \eta \partial_t \psi  - \int_{G_T}\Delta\eta \Delta\psi +\int_{G_T} F_{con,\delta}\psi +\int_{G_T} F\psi \nonumber\\
		&&= -\int_{\Omega^{\eta_{0,\delta}}}(\rho\bb{u})_{0,\delta}\cdot\boldsymbol\varphi(0) - \int_\Gamma v_0 \psi, \nonumber \\\label{momeqweak:delta}
    \end{eqnarray}
for all $\boldsymbol\varphi \in C_c^\infty([0,T)\times \mathbb{R}^3)$
and all $\psi\in C_c^\infty([0,T)\times \Gamma)$ such that $\boldsymbol\varphi(t,x,y,\eta(t,x,y))=\psi(t,x,y)\bb{e}_z$ on $G_T$ and $\boldsymbol{\varphi}=0$ on $\{z=0\}$. 
\subsubsection{Deriving the contact (in)equality and the corresponding estimates} \label{cont:sec}
Since $\eta_\delta\geq \delta>0$, we are allowed to test the coupled momentum equation $\eqref{momeqweak:delta}$ with $(\frac{z\psi}{\eta}\bb{e}_z,\psi)$ for any $\psi\in C_c^\infty([0,T)\times \Gamma)$ to obtain:
\begin{eqnarray*}
		&&\int_{Q_T^\eta} \rho \bb{u} \cdot \left(0,0, \frac{z\partial_t \psi}{\eta}- \frac{z\psi\partial_t \eta}{\eta^2}\right)   + \int_{Q_T^\eta}(\rho \bb{u} \otimes \bb{u}): \begin{bmatrix}
        0 &0 &0\\ 0 & 0 & 0 \\ 
         \frac{z\partial_x\psi}{\eta} - \frac{z\psi \partial_x \eta}{\eta^2}& \frac{z\partial_y \psi}{\eta} - \frac{z\psi \partial_y \eta}{\eta^2} &  \frac{\psi}\eta
		\end{bmatrix} \\
		&&\quad +\int_{Q_T^\eta} \rho^\gamma \frac{\psi}\eta 
		-  \int_{Q_T^\eta} \mathbb{S}( \nabla\bb{u}):  \begin{bmatrix}
         0 &0 &0\\ 0 & 0 & 0 \\ 
        \frac{z\partial_x \psi}{\eta} - \frac{z\psi \partial_x \eta}{\eta^2}& \frac{z\partial_y \psi}{\eta} - \frac{z\psi \partial_y \eta}{\eta^2} &  \frac{\psi}\eta
		\end{bmatrix} +\int_{G_T} \partial_t \eta \partial_t \psi - \int_{G_T}\Delta\eta \Delta\psi     \\
		&&\quad +\int_{G_T} F_{con,\delta}\psi +\int_{G_T} F\psi \\    
		&&=-\int_{\Omega^{\eta_{0,\delta}}}(\rho u_3)_0\frac{z\psi}{\eta_{0,\delta}} - \int_\Gamma v_0 \psi,
\end{eqnarray*}
where $(\rho u_3)_0 = (\rho \bb{u})_0 \cdot \bb{e}_z$. Now, we can write the convective term as
\begin{eqnarray*}
		&&\int_{Q_T^\eta} \rho \bb{u} \otimes \bb{u}:  \begin{bmatrix}
         0 &0 &0\\ 0 & 0 & 0 \\ 
        \frac{z\partial_x\psi}{\eta} - \frac{z\psi \partial_x \eta}{\eta^2}& \frac{z\partial_y \psi}{\eta} - \frac{z\psi \partial_y \eta}{\eta^2} &  \frac{\psi}\eta
		\end{bmatrix} \\
	&&= \int_{Q_T^\eta} \rho u_1 u_3\left(\frac{z\psi_x}{\eta} - \frac{z\psi 
    \partial_x \eta}{\eta^2} \right) + \int_{Q_T^\eta} \rho u_2 u_3\left(\frac{z\partial_y\psi}{\eta} - \frac{z\psi \partial_y \eta}{\eta^2} \right) + \int_{Q_T^\eta}\rho |u_3|^2 \frac{\psi}\eta .
\end{eqnarray*}
Next, the viscous term is written as
\begin{eqnarray*}
	&&\int_{Q_T^\eta} \mathbb{S}( \nabla\bb{u}):  \begin{bmatrix}
		0 &0 &0\\ 0 & 0 & 0 \\ 
	   \frac{z\partial_x\psi}{\eta} - \frac{z\psi \partial_x\eta}{\eta^2}& \frac{z\partial_y\psi}{\eta} - \frac{z\psi \partial_y \eta}{\eta^2} &  \frac{\psi}\eta
		\end{bmatrix} \\
	&&= \mu\int_{Q_T^\eta} (\partial_z u_1 + \partial_x u_3)\left(\frac{z\partial_x \psi}{\eta} - \frac{z\psi \partial_x \eta}{\eta^2} \right) +\mu\int_{Q_T^\eta} (\partial_z u_2 + \partial_y u_3)\left(\frac{z\partial_y \psi}{\eta} - \frac{z\psi \partial_y\eta}{\eta^2} \right) \\
    &&\quad+ 2\mu\int_{Q_T^\eta} \partial_z u_3 \frac{\psi}\eta + \left(\lambda - \frac{2\mu}3\right)\int_{Q_T^\eta} \left(\partial_x u_1 + \partial_y u_2 + \partial_z u_3 \right) \frac{\psi}\eta \\
    && = \mu\int_{Q_T^\eta} (\partial_z u_1 + \partial_x u_3)\left(\frac{z\partial_x \psi}{\eta} - \frac{z\psi \partial_x \eta}{\eta^2} \right) +\mu\int_{Q_T^\eta} (\partial_z u_2 + \partial_y u_3)\left(\frac{z\partial_y \psi}{\eta} - \frac{z\psi \partial_y\eta}{\eta^2} \right)\\
    &&\quad + \left(\frac{4\mu}3 +\lambda \right)\int_{Q_T^\eta} \partial_z u_3 \frac{\psi}\eta - \left(\lambda - \frac{2\mu}3\right)\int_{Q_T^\eta} \left( u_1 \left(\frac{\partial_x \psi}{\eta}-\frac{\psi \partial_x \eta}{\eta^2}\right)+   
    u_2 \left(\frac{\partial_y \psi}{\eta}-\frac{\psi \partial_y \eta}{\eta^2}\right)    \right)
\end{eqnarray*}
by partial integration, where we used $u_1,u_2 \in L^2(0,T; H_0^1(\Omega^\eta))$, and since
\begin{eqnarray*}
	\int_{Q_T^\eta} \partial_z u_3 \frac{\psi}\eta = \int_{G_T} \int_0^{\eta} \partial_z u_3 \frac{\psi}\eta = \int_{G_T} u_3\big|_{z=0}^{z=\eta} \frac{\psi}\eta = \int_{G_T} \frac{\partial_t \eta}{\eta} \psi = \int_{G_T} \partial_t \ln\eta ~\psi,
\end{eqnarray*}
we finally obtain
\begin{eqnarray}
	&&\int_{G_T}F\psi + \int_{G_T} F_{con,\delta}\psi + \int_{Q_T^\eta} \rho^\gamma \frac{\psi}\eta +\int_{Q_T^\eta}\rho |u_3|^2 \frac{\psi}\eta -  \left(\frac{4\mu}3 +\lambda \right)\int_{G_T} \partial_t \ln\eta ~\psi\nonumber \\
	&&=-\int_{Q_T^\eta} \rho u_3 \left( \frac{z\partial_t \psi}{\eta}- \frac{z\psi\partial_t \eta}{\eta^2} \right)  -\int_{Q_T^\eta} \rho u_1 u_3\left(\frac{z\partial_x \psi}{\eta} - \frac{z\psi \partial_x \eta}{\eta^2} \right)\nonumber \\
    &&\quad- \int_{Q_T^\eta} \rho u_2 u_3\left(\frac{z\partial_y \psi}{\eta} - \frac{z\psi \partial_y \eta}{\eta^2} \right) \nonumber\\
	&&\quad+\mu\int_{Q_T^\eta} (\partial_z u_1 + \partial_x u_3)\left(\frac{z\psi_x}{\eta} - \frac{z\psi \eta_x}{\eta^2} \right) + \mu\int_{Q_T^\eta} (\partial_z u_2 + \partial_y u_3)\left(\frac{z\psi_y}{\eta} - \frac{z\psi \eta_y}{\eta^2} \right)  \nonumber\\
    &&\quad - \left(\lambda - \frac{2\mu}3\right)\int_{Q_T^\eta} \left( u_1 \left(\frac{\partial_x \psi}{\eta}-\frac{\psi \partial_x \eta}{\eta^2}\right)+   
    u_2 \left(\frac{\partial_y \psi}{\eta}-\frac{\psi \partial_y \eta}{\eta^2}\right)    \right) \nonumber\\
	&&\quad+\int_{G_T}\Delta\eta \Delta\psi -\int_{G_T} \partial_t \eta \partial_t \psi \nonumber \\
	&&\quad-\int_{\Omega^{\eta_{0,\delta}}}(\rho u_3)_0\frac{z\psi(0)}{\eta_{0,\delta}} - \int_\Gamma v_0 \psi(0). \label{eq:contact}
\end{eqnarray}
For a.a. $t\in (0,T)$, choosing $\psi = \chi_{\{0,t\}}$ by the density argument leads to
\begin{eqnarray}
	&&\int_0^t \int_\Gamma F^+ + \int_0^t \int_\Gamma F_{con,\delta}+\int_0^t \int_{\Omega^\eta } \frac{\rho^\gamma }\eta +\int_0^t \int_{\Omega^\eta} \frac{\rho |u_3|^2}\eta +\left(\frac{4\mu}3 +\lambda \right) \int_{\Gamma} \ln\eta^-  \nonumber\\
	&&=\int_0^t \int_{\Omega^\eta } \rho u_3 z \frac{\partial_t \eta}{\eta^2}  +\int_0^t \int_{\Omega^\eta } \rho u_1 u_3 \frac{z\partial_x \eta}{\eta^2} + \int_0^t \int_{\Omega^\eta } \rho u_2 u_3\frac{z \partial_y \eta}{\eta^2}\nonumber \\
	&&\quad-\mu\int_0^t \int_{\Omega^\eta } (\partial_z u_1 + \partial_x u_3) \frac{z\partial_x \eta}{\eta^2}  -\mu\int_0^t \int_{\Omega^\eta } (\partial_z u_2 + \partial_y u_3) \frac{z \partial_y \eta}{\eta^2} \nonumber \\
    &&\quad +\left(\lambda - \frac{2\mu}3\right)\int_{Q_T^\eta} \left( u_1 \frac{ \partial_x \eta}{\eta^2}+   
    u_2 \frac{\partial_y \eta}{\eta^2}    \right) \nonumber\\
	&&\quad + \left(\frac{4\mu}3 +\lambda \right)\int_{\Gamma} \ln\eta^+  - \left(\frac{4\mu}3 +\lambda \right)\int_{\Gamma} \ln \eta_{0,\delta} +\int_0^t \int_\Gamma F^-  \nonumber\\
	&&\quad +\int_{\Omega^\eta(t)}\rho(t)u_3(t) \frac{z}{\eta(t)}-\int_{\Omega^{\eta_{0,\delta}}}(\rho u_3)_{0,\delta}\frac{z}{\eta_{0,\delta}} +\int_\Gamma \partial_t \eta(t) -\int_\Gamma v_0.\label{ineq:1} 
\end{eqnarray}
In order to obtain new estimates, it is enough to bound the right-hand side as the left-hand side is positive. We start with
\begin{eqnarray}
	&&\left|\int_0^t \int_{\Omega^\eta} \rho u_3  z \frac{\partial_t \eta}{\eta^2} \right|
	\leq \|\rho\|_{L^\infty(0,T; L^\gamma(\Omega^\eta))} \|\bb{u}\|_{L^2(0,T; L^p(\Omega^\eta))} \left\|\frac{\partial_t \eta}\eta \right\|_{L^2(0,T; L^2(\Omega^\eta))}\nonumber \\
	&& \leq C\|\rho\|_{L^\infty(0,T; L^\gamma(\Omega^\eta))}^\gamma + \|\bb{u}\|_{L^2(0,T; L^p(\Omega^\eta))}^p +\int_0^T\int_\Gamma\underbrace{\int_0^\eta \frac{|\partial_t \eta|^2}{\eta^2}}_{=\frac{|\partial_t \eta|^2}{\eta}} \nonumber\\
	&&\leq C+ C(1+\|\eta\|_{L^\infty((0,T)\times \Gamma)})^p\|\bb{u}\|_{L^2(0,T; H^1(\Omega^\eta))}^p + \|\partial_z u_3\|_{L^2(0,T; L^2(\Omega^\eta))}^2\nonumber\\
		&&\leq C, \label{bound1}
\end{eqnarray}
for $\gamma>3$ and $p = \frac{2\gamma}{\gamma-2}<6$, where we used $\frac{z}\eta \leq 1$ and Lemma $\ref{lem:etat}$. Next, noticing that
\begin{eqnarray*}
	\|\sqrt{\rho} \bb{u}\|_{L^2(0,T; L^{\frac{2p\gamma}{p+2\gamma}}(\Omega^\eta))} \leq \|\sqrt{\rho}\|_{L^\infty(0,T; L^{2\gamma}(\Omega^\eta))} \| \bb{u}\|_{L^2(0,T; L^p(\Omega^\eta))} \leq C, 
\end{eqnarray*}
we can bound
\begin{eqnarray}
	&&\left|\int_0^t \int_{\Omega^\eta}\rho u_1 u_3 \frac{z \partial_x \eta}{\eta^2} + \int_0^t \int_{\Omega^\eta} \rho u_2 u_3\frac{z \partial_y \eta}{\eta^2}\right| \nonumber\\
	&&\leq \int_0^t \int_{\Omega^\eta} \frac{z}{\eta}\frac{\sqrt{\rho}|u_3|}{\sqrt{\eta}}\sqrt{\rho} |\bb{u}| \frac{|\nabla\eta|^{\frac23}}{\sqrt{\eta}} |\nabla\eta|^{\frac13} \nonumber\\
	&&\leq \left(\int_0^t \int_{\Omega^\eta} \frac{\rho |u_3|^2}{\eta} \right)^{\frac12}  \underbrace{\|\sqrt{\rho} \bb{u}\|_{L^2(0,T; L^{\frac{2p\gamma}{p+2\gamma}}(\Omega^\eta))}}_{\leq C} \nonumber\\
    &&\quad \times\Bigg\|\Bigg(\int_\Gamma \underbrace{\int_0^\eta 
	\frac{|\nabla\eta|^4}{\eta^3}}_{= \frac{|\nabla\eta|^4}{\eta^2}}\Bigg)^{\frac16}\Bigg\|_{L^\infty(0,T)} \underbrace{\| ~|\nabla\eta|^{\frac13}\|_{L^\infty(0,T; L^q(\Omega^\eta))}}_{\leq C\|\Delta \eta\|_{L^\infty(0,T; L^2(\Gamma))}^{\frac13}} \nonumber \\
	&&\leq \frac12\int_0^t \int_{\Omega^\eta}\frac{\rho |u_3|^2}{\eta} + C,  \label{bound2}
\end{eqnarray}
for $\gamma>3$ and $6>p> \frac{2\gamma}{\gamma-2}$ and $q=q(p,\gamma)$ large enough, where we used Lemma $\ref{eta:n}$ with $s=2$. We proceed to the viscous term
    \begin{eqnarray}
		&&\left|\mu\int_0^t \int_{\Omega^\eta}(\partial_z u_1 + \partial_x u_3) \frac{z\eta_x}{\eta^2}+ \mu\int_0^t \int_{\Omega^\eta} (\partial_z u_2 + \partial_y u_3) \frac{z \eta_y}{\eta^2}\right|\nonumber \\ 
        &&\leq \left(\int_{Q_T^\eta} \mathbb{S}(\nabla\bb{u}):\nabla \bb{u} \right) \left(\int_0^t \int_{\Omega^\eta} \frac{|\nabla\eta|^2}{\eta^2} \right)^{1/2}\nonumber\\
		&&\leq C \left(\int_0^t\int_{\Gamma} \int_0^{\eta} \frac{|\nabla\eta|^2}{\eta^2} \right)^{1/2} = C \left( \int_0^t\int_{\Gamma}  \frac{|\nabla\eta|^2}{\eta} \right)^{1/2}\nonumber \\
		&&\leq C \Big(\underbrace{\int_0^t\int_{\Gamma} |\Delta\eta|}_{\leq Ct}\Big)^{1/2} \leq C, \label{bound3}
	\end{eqnarray}
by Lemma $\ref{eta:n}$ for $s=1$, while similarly
\begin{eqnarray}
    &&\left(\lambda - \frac{2\mu}3\right)\int_{Q_T^\eta} \left( u_1 \frac{ \partial_x \eta}{\eta^2}+   
    u_2 \frac{\partial_y \eta}{\eta^2}    \right) \nonumber \\
    &&\leq \left(\lambda - \frac{2\mu}3\right)\left\| \frac{\bb{u}}{\eta}\right\|_{L^2(Q_T^\eta)} \left(\int_0^t \int_{\Omega^\eta} \frac{|\nabla\eta|^2}{\eta^2} \right)^{1/2} \leq C, \label{bound4}
\end{eqnarray}
by Lemma $\ref{lem:etat}$. Since $\ln\eta^+ \leq \eta$, while the rest of the terms can be bounded similarly, we obtain that the right-hand side of $\eqref{ineq:1}$ is bounded. This implies the uniform boundedness
\begin{eqnarray*}
	\ln\eta \in C([0,T]; L^1(\Gamma)), \quad \frac{\rho^\gamma}\eta, \frac{\rho |u_3|^2}{\eta} \in L^1(Q_T^\eta).
\end{eqnarray*}
	
\subsection{Passing to the limit $\delta\to 0$}
Let us denote the approximate solutions to the problem with contact force as $(\rho_\delta,\bb{u}_\delta,\eta_\delta)$ and the weak limits corresponding to a converging subsequence as:
    \begin{eqnarray*}
         & \eta_\delta \rightharpoonup \eta& \quad \text{ weakly* in }\quad W^{1,\infty}(0,T; L^2(\Gamma))\cap L^\infty(0,T; H^2(\Gamma)),\\
        &\rho_\delta\chi_{|\Omega^{\eta_\delta}} \rightharpoonup \rho \chi_{|\Omega^{\eta}}&\quad \text{ weakly* in }\quad L^\infty(0,T; L^\gamma( \mathbb{R}^3)), \\
         &\bb{u}_\delta\chi_{|\Omega^{\eta_\delta}} \rightharpoonup \bb{u}\chi_{|\Omega^{\eta}}&\quad \text{ weakly in }\quad L^2((0,T)\times \mathbb{R}^3), \\
        &\nabla\bb{u}_\delta\chi_{|\Omega^{\eta_\delta}} \rightharpoonup \nabla\bb{u}\chi_{|\Omega^{\eta}}&\quad \text{ weakly in }\quad L^2((0,T)\times \mathbb{R}^3).
    \end{eqnarray*}
At this point, the convergence $\delta\to 0$ is similar to previous convergences, since most of arguments can still be localized. However, as the weak convergence of the pressure depends on the domain, this part is different, as the fluid domain can degenerate. Surprisingly so, the estimate $\frac{\rho_{\delta}^\gamma}{\eta_\delta}\in L^1(Q_T^{\eta_\delta})$ actually ensures that there will be no pressure concentrations:\footnote{It is actually enough to prove that pressure converges weakly away from contact, as test functions in $\eqref{momeqweak}$ vanish on $\{z=0\}$. However, the following result is another useful consequence of the contact inequality worth mentioning that can have further applications, as it tells us that the pressure near contact region is in a sense regularized by the plate.}
\begin{lem}\label{weak:conv:pr:delta}
There exists $\overline{p}\in L^1(Q_T^\eta)$ such that 
    \begin{eqnarray*}
     \lim_{\delta\to 0}\int_{Q_T^{\eta_{\delta}}}\rho_{\delta}^\gamma \varphi \to \int_{Q_T^\eta}\overline{p} \varphi
    \end{eqnarray*}
for any $\varphi\in C^\infty([0,T]\times \mathbb{R}^3)$.
\end{lem}
\begin{proof}
Fix $\varphi\in C^\infty([0,T]\times \mathbb{R}^3)$. First, since $\eta_{\delta} \to \eta$  in  $C^{0,\alpha}([0,T]\times \Gamma)$, for every $(t,x,y)\in G_T$ such that $\eta(t,x,y)>0$ there exists an open set $(t,x,y)\in Q:=I\times O\subset G_T$ and $\delta_Q>0$ such that $\eta_\delta\geq c_Q>0$ on $Q$ for every $\delta<\delta_Q$. We can thus localize the proof of Lemma $\ref{weak:conv:pr}$ on 
$Q\times\mathbb{R}^+\cap Q_T^{\eta_\delta}$ to obtain that
\begin{eqnarray*}
     \lim_{\delta\to 0}\int_I\int_{O}\int_{0}^{\eta_\delta}\rho_{\delta}^\gamma \varphi \to \int_I\int_{O}\int_{0}^{\eta}\overline{p} \varphi,
\end{eqnarray*}
where $\overline{p}\in L^1(Q_T^\eta)$ is once again determined locally as a weak limit of $\rho_{\delta}^\gamma$. As the choice of $(t,x,y)$ such that $\eta(t,x,y)>0$ was arbitrary, we can thus conclude that for every $\kappa>0$
\begin{eqnarray*}
    \lim_{\delta\to 0}\int_0^T\int_{\Gamma\cap \{\eta_\delta\geq \kappa\}}\int_{0}^{\eta_\delta}\rho_{\delta}^\gamma \varphi \to \int_0^T\int_{\Gamma\cap \{\eta\geq \kappa\}}\int_{0}^{\eta}\overline{p} \varphi.
\end{eqnarray*}
On the other hand, since $\frac{\rho_\delta^\gamma}{\eta_\delta}$ is uniformly bounded in $L^1(Q_T^{\eta_\delta})$ by $\eqref{ineq:1}$, one has
\begin{eqnarray*}
    \int_0^T\int_{\Gamma\cap \{\eta_\delta< \kappa\}}\int_{0}^{\eta_\delta}\rho_{\delta}^\gamma \varphi = \int_0^T\int_{\Gamma\cap \{\eta_\delta< \kappa\}}\int_{0}^{\eta_\delta}\frac{\rho_{\delta}^\gamma}{\eta_\delta} \eta_\delta \varphi \leq C\kappa,
\end{eqnarray*}
so
\begin{eqnarray*}
    &&\lim_{\delta\to 0}\int_{Q_T^{\eta_{\delta}}} \rho_{\delta}^\gamma \varphi  \\
    &&=\underbrace{\lim_{\delta\to 0}\int_0^T\int_{\Gamma\cap \{\eta_\delta\geq \kappa\}}\int_{0}^{\eta_\delta}\rho_{\delta}^\gamma \varphi}_{=\int_0^T\int_{\Gamma\cap \{\eta\geq \kappa\}}\int_{0}^{\eta}\overline{p} \varphi}  +\underbrace{\lim_{\delta\to 0}\int_0^T\int_{\Gamma\cap \{\eta_\delta< \kappa\}}\int_{0}^{\eta_\delta}\rho_{\delta}^\gamma \varphi}_{\leq C\kappa}.
\end{eqnarray*}
Passing to the limit $\kappa\to 0$, the conclusion follows.
\end{proof}
Finally, It is worth noting that the weak limit of contact force $F_{con,\delta}$ will be supported on $\{z=0\}$ and will not appear in $\eqref{momeqweak}$ as the test functions in $\eqref{momeqweak}$ vanish on $\{z=0\}$. 
\begin{rem}
    The limit of $F_{con,\delta}$ can actually be included in the left-hand side of the contact inequality $\eqref{ver:ineq}$ along with other defect measures. However, as measures lose their connection with original functions throughout the limit, it is not clear what part of the the defect would belong to contact force. Moreover, as the vertical fluid dissipation is quite strong, it may happen that $F_{con,\delta}\to 0$ as $\delta \to 0$. Whether or not the usage of contact force in the construction can be avoided by utilizing the vertical fluid dissipation is an interesting question. 
\end{rem}

\subsubsection{Passing to the limit in the contact inequality}
First, since $\ln\eta_\delta\in C([0,T]; L^1(\Gamma))$, while $\eta_\delta \to \eta \in C^{0,\alpha}([0,T]\times \Gamma)$ for any $\alpha\in (0,1/3)$, by Lemma \ref{ln:lemma} we obtain $\ln\eta \in C([0,T]; L^1(\Gamma))$ and
\begin{eqnarray*}
    \int_\Gamma \ln \eta(t)^- \leq \lim_{\delta\to 0} \int_\Gamma\ln\eta_\delta(t)^-, \quad \text{ for every } t\in [0,T].
\end{eqnarray*}
Next, for $b>0$, we have that $\rho_\delta^\gamma \to \rho^\gamma$ and $\frac1{\eta_\delta} \to \frac1\eta$ a.e. on $Q_T^\eta\cap\{z\geq b\}$, which together imply
	\begin{eqnarray*}
		\int_0^t \int_\Gamma \int_b^{\eta} \frac{\rho^\gamma}{\eta} = \lim_{\delta\to 0} \int_0^t \int_\Gamma \int_b^{\eta_\delta} \frac{\rho_\delta^\gamma}{\eta_\delta} \leq \lim_{\delta\to 0}  \int_0^t\int_{\Omega^{\eta_\delta}}\frac{\rho_\delta^\gamma}{\eta_\delta},
	\end{eqnarray*}
so passing to the limit $b\to 0$ gives us 
	\begin{eqnarray*}
        \int_0^t\int_{\Omega^\eta} \frac{\rho^\gamma}{\eta}  \leq \lim_{\delta\to 0} \int_0^t\int_{\Omega^{\eta_\delta}}\frac{\rho_\delta^\gamma}{\eta_\delta}.
	\end{eqnarray*}
In order to pass to the limit on the right-hand side of $\eqref{eq:contact}$, let us prove that all the integrands are bounded in $L^p(Q_T^\eta)$ for some $p>1$, then we can identify the terms locally. We start with the convective term. It can be estimated in the same way as in $\eqref{bound2}$ but with larger $q$ to obtain
	\begin{eqnarray*}
		\rho u_{\delta,1} u_{\delta,3} \frac{z \partial_x \eta_{\delta}}{\eta_\delta^2} + \rho_\delta u_{\delta,2} u_{\delta,3}\frac{z \partial_y \eta_{\delta}}{\eta_\delta^2} \in L^1(0,T; L^p(\Omega^{\eta_\delta}(t))),
	\end{eqnarray*}
for some $p>1$ (but close to). In order to gain more time integrability, we can interpolate between 
	\begin{eqnarray*}
		\|\sqrt{\rho_\delta} \bb{u}_\delta\|_{L^2(0,T; L^{\frac{2p\gamma}{p+2\gamma}}(\Omega^{\eta_\delta}))} \leq \|\sqrt{\rho_\delta}\|_{L^\infty(0,T; L^{2\gamma}(\Omega^{\eta_\delta}))} \| \bb{u}_\delta\|_{L^2(0,T; L^p(\Omega^{\eta_\delta}))} \leq C
	\end{eqnarray*}
	and
	\begin{eqnarray*}
		\|\sqrt{\rho_\delta} \bb{u}_\delta\|_{L^\infty(0,T; L^{2}(\Omega^{\eta_\delta}))} \leq C
	\end{eqnarray*}
to obtain for $\theta \in (0,1)$
	\begin{eqnarray*}
		\|\sqrt{\rho_\delta} \bb{u}_\delta\|_{L^p(0,T; L^{q}(\Omega^{\eta_\delta}))} \leq \|\sqrt{\rho_\delta} \bb{u}_\delta\|_{L^2(0,T; L^{\frac{2p\gamma}{p+2\gamma}}(\Omega^{\eta_\delta}))}^{1-\theta} \|\sqrt{\rho_\delta} \bb{u}_\delta\|_{L^\infty(0,T; L^{2}(\Omega^{\eta_\delta}))}^\theta\leq C,
	\end{eqnarray*}
for some $p>2$ and $q<\frac{12\gamma}{6+2\gamma}$ (but close to), which is obtained by choosing $\theta$ close to zero. Then, with this better bound, estimating as in $\eqref{bound2}$ gives us
	\begin{eqnarray*}
		\rho u_{\delta,1} u_{\delta,3} \frac{z \partial_x \eta_{\delta}}{\eta_\delta^2} + \rho_\delta u_{\delta,2} u_{\delta,3}\frac{z \partial_y \eta_{\delta}}{\eta_\delta^2} \in L^r(Q_T^{\eta_\delta})
	\end{eqnarray*}
for some $r>1$. Similarly, for $\eqref{bound1}$, we can interpolate between
	\begin{eqnarray*}
		\|\rho_\delta \bb{u}_\delta\|_{L^2(0,T; L^{\frac{p\gamma}{\gamma+p}}(\Omega^{\eta_\delta}))} \leq  \|\rho_\delta\|_{L^\infty(0,T; L^\gamma(\Omega^{\eta_\delta}))} \|\bb{u}_\delta\|_{L^2(0,T; L^p(\Omega^{\eta_\delta}))} \leq C
	\end{eqnarray*}
and
	\begin{eqnarray*}
		\|\rho_\delta \bb{u}_\delta\|_{L^\infty(0,T; L^{\frac{2\gamma}{\gamma+1}}(\Omega^{\eta_\delta}))} \leq \|\sqrt{\rho_\delta}\|_{L^\infty(0,T; L^{2\gamma}(\Omega^{\eta_\delta}))} \|\sqrt{\rho_\delta} \bb{u}_\delta\|_{L^\infty(0,T; L^{2}(\Omega^{\eta_\delta}))} \leq C
	\end{eqnarray*}
	to obtain for $\theta\in (0,1)$
	\begin{eqnarray*}
		\|\rho_\delta \bb{u}_\delta\|_{L^s(0,T; L^{r}(\Omega^{\eta_\delta}))}\leq \|\rho_\delta \bb{u}_\delta\|_{L^2(0,T; L^{\frac{p\gamma}{\gamma+p}}(\Omega^{\eta_\delta}))}^{1-\theta} \|\rho_\delta \bb{u}_\delta\|_{L^\infty(0,T; L^{\frac{2\gamma}{\gamma+1}}(\Omega^{\eta_\delta}))}^{\theta}  \leq C,
	\end{eqnarray*}
for some $s>2$ and $r<\frac{6\gamma}{\gamma+6}$ (but close to), which is once again obtained by $\theta$ close to zero. Then, estimating as in $\eqref{bound1}$, we conclude
	\begin{eqnarray*}
		\rho_\delta u_{\delta,3}  z \frac{\partial_t \eta_{\delta}}{\eta_\delta^2}\in L^p(Q_T^{\eta_\delta})
	\end{eqnarray*}
for some $p>1$. As other terms are easier to bound, we conclude that all the integrands on the right-hand side of $\eqref{eq:contact}$ converge to functions (and not measures). \\
	
Now, in order to pass to the limit $\delta\to 0$, it is enough to identify each of the limiting integrands on every open set $Q=I\times O \subset Q_T^\eta$, where $I\subset (0,T)$ and $O \subset \Omega^\eta(t)$ for all $t\in I$. Fix such $Q$. First note that $\eta>0$ on $Q$, so by the convergence $\eta_\delta \to \eta$ in $C^{0,\alpha}([0,T]\times \Gamma)$ for any $\alpha\in (0,1/3)$, there is an $\delta_Q>0$ such that for all $\delta<\delta_Q$ we have $O\subset \Omega^{\eta_\delta}(t)$ for every $t\in I$ and $\eta_\delta\geq a>0$ on $Q$. As a direct consequence, one also has $\frac{1}\eta_\delta \to \frac1\eta$ in $L^\infty(Q)$. Moreover, by the uniform estimates and Aubin-Lions lemma, we have $\nabla\eta_\delta\to \nabla \eta$ on $L^\infty(I; L^p(O))$ for any $p<\infty$. Now, as in $\eqref{convrhouu}$
	\begin{eqnarray*}
		\rho_\delta \bb{u}_\delta\otimes\bb{u}_\delta \rightharpoonup \rho \bb{u} \otimes \bb{u} \quad \text{weakly in } L^q(Q)
	\end{eqnarray*}
for some $q>1$, so 
	\begin{eqnarray*}
		&&\rho_\delta \bb{u}_\delta\otimes\bb{u}_\delta:\nabla \left(\frac{z \bb{e}_z}{\eta_\delta} \right) \\
		&&=\rho u_{\delta,1} u_{\delta,3} \frac{z 
		\partial_x\eta_{\delta}}{\eta_\delta^2} + \rho u_{\delta,2} u_{\delta,3}\frac{z \partial_y \eta_{\delta}}{\eta_\delta^2}\rightharpoonup \rho u_1 u_3 \frac{z \partial_x \eta}{\eta^2} + \rho u_2 u_3\frac{z \partial_y \eta}{\eta^2}\quad \text{weakly in } L^r(Q)
	\end{eqnarray*}
for $q>r>1$. Next, from the continuity equation, we have
	\begin{eqnarray*}
		\partial_t \rho_{\delta} = - \nabla\cdot (\rho_\delta \bb{u}_\delta) \in L^2(I; H^{-1}(O))
	\end{eqnarray*}
so $\rho_\delta \in L^\infty(I; L^\gamma(O))\cap H^1(I; H^{-1}(O))\hookrightarrow\hookrightarrow L^2(I; H^{-1}(O))$, and since $\bb{u}_\delta \in L^2(I; H^1(O))$, we conclude $\rho_\delta \bb{u}_\delta \rightharpoonup \rho\bb{u}$ weakly in $L^2(I; L^{\frac{p\gamma}{p+\gamma}}(O))\hookrightarrow L^2(I; L^{2}(O))$, for $\gamma>3$ and $p = \frac{\gamma-2}{2\gamma}<6$. From the momentum equation, we have $\rho_\delta \bb{u}_\delta \in H^1(I; H^{-3}(O))$, while $\rho_\delta \bb{u}_\delta\in L^2(I; L^{2}(O))$ so $\rho_\delta\bb{u}_\delta \to \rho\bb{u}$ in $L^2(I; H^{-s}(O))$, for any $s>0$. Now, from the trace estimates given in Lemma $\ref{trace:lemma}$, we have $\partial_t \eta_{\delta}\in L^2(0,T; H^s(\Gamma))$, so $\partial_t \eta_{\delta}\rightharpoonup \partial_t \eta$ weakly in $L^2(0,T; H^s(\Gamma))$. Therefore, combining the above convergences one has
	\begin{eqnarray*}
		-\rho_\delta\bb{u}_\delta \cdot \partial_t \left(\frac{z \bb{e}_z}{\eta_\delta} \right) = \rho_\delta u_{\delta,3}  z \frac{\partial_t \eta_{\delta}}{\eta_\delta^2} \rightharpoonup \rho u_3  z \frac{\partial_t \eta}{\eta^2} \quad \text{ weakly in } L^1(Q).
	\end{eqnarray*}
The remaining convergences can be dealt with in a simpler fashion so we omit them. Thus, the proof of Theorem $\ref{main1}$ is complete.

\section{Detachment of all contact by pressure - proof of Theorem $\ref{main2}$}
\subsection{Case $F\in L^1(0,\infty; L^2(\Gamma))$}
Let $(\rho,\bb{u},\eta)$ be a given weak solution in the sense of Definition $\ref{weak:sol:def}$ and let us assume that $\min_{(x,y)\in\Gamma}\eta(t,x,y)=0$ for all $t$. First, note that for any function $f\geq 0$ with $f\in H^2(\Gamma)$ satisfies
	\begin{eqnarray}
		\max_{x\in \Gamma}f(x) \leq \min_{x\in \Gamma} f(x) + \max_{x,y \in \Gamma, x\neq y}\frac{|f(x)-f(y)|}{|x-y|^\alpha}\left|\frac{\text{diam}(\Gamma)}2\right|^\alpha  \leq \min_{x\in \Gamma} f(x)+ C\|\Delta f\|_{L^2(\Gamma)}\label{max:est}
	\end{eqnarray}
for any $\alpha\in (0,1/2)$. Therefore, since $\|\Delta\eta\|_{L^\infty(0,T; L^2(\Gamma))}\leq C$, there is a constant $H>0$ such that $\|\eta\|_{L^\infty((0,T)\times\Gamma)}\leq H$, which combined with uniform estimates coming from the energy inequality $\eqref{en:ineq}$, one has by Lemma $\ref{imbed}$ and Lemma $\ref{H1}$
	\begin{eqnarray*}
		\|\bb{u}\|_{L^2(0,T; L^p(\Omega^\eta))}\leq  C\|\bb{u}\|_{L^2(0,T; H^1(\Omega^\eta))} \leq C, \quad \text{for any } p\in [1,6),
	\end{eqnarray*}
where $C$ only depends on $p,\Gamma$ and initial energy. We will show that this will give us a contradiction.\\

Repeating the calculation as in $\eqref{bound1},\eqref{bound2},\eqref{bound3},\eqref{bound4}$ to bound the terms on the right-hand side of $\eqref{ver:ineq}$ and taking into consideration that we assumed $\|\eta\|_{L^\infty((0,T)\times\Gamma)}\leq H$, we obtain that 
	\begin{eqnarray*}
		\int_0^T\int_{\Omega^\eta}  \frac{\rho^\gamma}\eta \leq C(1+\sqrt{T}) - \int_0^T \int_\Gamma F. 
	\end{eqnarray*}
Now, in order to bound from below the left-hand side term, let us first calculate
	\begin{eqnarray*}
		&&m=\int_{\Omega^\eta(t)} \rho(t) \leq \left(\int_{\Omega^\eta(t)} 1^{\frac{\gamma}{\gamma-1}} \right)^{\frac{\gamma-1}{\gamma}} \left(\int_{\Omega^\eta(t)} \rho^\gamma(t) \right)^{\frac1\gamma} \leq \left(\int_{\Gamma} |\eta(t)|_\infty \right)^{\frac{\gamma-1}{\gamma}} \left(\int_{\Omega^\eta(t)} \rho^\gamma (t)\right)^{\frac1\gamma} \\
        &&=  (|\Gamma|~  |\eta(t)|_\infty)^{\frac{\gamma-1}{\gamma}}\left(\int_{\Omega^\eta(t)} \rho^\gamma(t) \right)^{\frac1\gamma}
	\end{eqnarray*}
for a.a. $t\in (0,T)$, where $|\eta(t)|_\infty:= \max_{(x,y)\in \Gamma} \eta(t,x,y)$, so
	\begin{eqnarray*}
		\int_{\Omega^\eta(t)} \rho^\gamma (t)  \geq \frac{m^\gamma}{(|\Gamma|~ |\eta(t)|_\infty)^{\gamma-1}}.
	\end{eqnarray*}
Therefore
	\begin{eqnarray*}
		\int_{\Omega^\eta(t)} \frac{ \rho^\gamma}\eta (t) \geq \frac{1}{|\eta(t)|_\infty}\int_{\Omega^\eta(t)} \rho^\gamma (t)\geq \frac{1}{|\eta(t)|_\infty}\frac{m^\gamma}{(|\Gamma|~ |\eta(t)|_\infty)^{\gamma-1}} = \frac{m^\gamma}{|\Gamma|^{\gamma-1}|\eta(t)|_\infty^{\gamma}}.
	\end{eqnarray*}
Combining the above estimates, by integrating over $(0,T)$, we obtain
	\begin{eqnarray*}
		&&\frac{1}{\max_{t\in(0,T)} |\eta(t)|_\infty^{\gamma}}\frac{Tm^\gamma}{|\Gamma|^{\gamma-1}} = \frac{1}{\max_{t\in(0,T)} |\eta(t)|_\infty^{\gamma}}\int_0^T \frac{m^\gamma}{|\Gamma|^{\gamma-1}}\\
		&&\leq \int_0^T\frac{m^\gamma}{|\Gamma|^{\gamma-1}|\eta(t)|_\infty^{\gamma}}\leq \int_0^T\int_{\Omega^\eta}  \frac{\rho^\gamma}\eta(t) \leq C(1+\sqrt{T})-\int_0^T\int_{\Gamma} F,
	\end{eqnarray*}
so
	\begin{eqnarray}
		\max_{t\in(0,T)} |\eta(t)|_\infty\geq && \left(\frac{Tm^\gamma}{|\Gamma|^{\gamma-1} C(1+\sqrt{T}) - |\Gamma|^{\gamma-1}\int_0^T \int_{\Gamma} F}\right)^{\frac1\gamma} \nonumber \\
		&&= \left(\frac{m^\gamma}{|\Gamma|^{\gamma-1} C\frac{(1+\sqrt{T})}{T} -|\Gamma|^{\gamma-1} \frac1T\int_0^T \int_{\Gamma} F}\right)^{\frac1\gamma} \to \infty \label{blowup}
	\end{eqnarray}
as $T\to \infty$, which is a contradiction. Consequently, there is a $T>0$ large enough such that $\eta(T)>0$. In particular, this implies that if we start from initial data with contact, at some point it will break.
	
	\begin{rem}
		Alternatively, we can avoid using the contradiction argument in this proof by tracking the dependence of the right-hand side of $\eqref{ver:ineq}$ on $\|\eta\|_{L^\infty((0,T)\times\Gamma)}$. This additional dependence would then increase the power of $\max_{t\in(0,T)} |\eta(t)|_\infty$ on the left-hand side of $\eqref{blowup}$, implying also that $ |\eta(t)|_\infty \nearrow \infty$ as $t\to\infty$ since the force $F\in L^1(0,\infty; L^2(\Gamma))$ loses its strength over time. On the other hand, by using the contradiction argument, the amount of detail is greatly reduced.
	\end{rem}

\subsection{Case $F\in L^1(\Gamma)$}	
Once again, let us assume that $\min_{(x,y)\in\Gamma}\eta(t,x,y)=0$ for all $t>0$. Then, due to $\eqref{max:est}$ we have $\|\eta(t)\|_{L^\infty(\Gamma)}\leq C\|\Delta \eta(t)\|_{L^2(\Gamma)}$. This means that we can obtain uniform estimates from the energy inequality $\eqref{en:ineq}$ by first estimating
\begin{eqnarray}
    &&\int_0^t \int_\Gamma F\partial_t \eta = \int_\Gamma F \eta(t) - \underbrace{\int_\Gamma F \eta_0}_{:=C_0} \nonumber \\
    &&\leq \|F\|_{L^1(\Gamma)} \|\eta(t)\|_{L^\infty(\Gamma)} - C_0 \nonumber\\
    &&\leq C \|F\|_{L^1(\Gamma)} \|\Delta \eta(t)\|_{L^2(\Gamma)} - C_0 \nonumber\\
    &&\leq\frac{C^2}2\|F\|_{L^1(\Gamma)}^2 + \frac12 \|\Delta \eta(t)\|_{L^2(\Gamma))} - C_0 \label{l1:case}
\end{eqnarray}
and then absorbing $\frac12\|\Delta \eta(t)\|_{ L^2(\Gamma)}$. Now, by assuming, say, $\|F\|_{L^1(\Gamma)}\leq E_0$, where 
\begin{eqnarray*}
    E_0:=\displaystyle{\int_{\Omega^{\eta_0}} \left( \frac{1}{2\rho_0} |(\rho\bb{u})_0|^2 + \frac{\rho_0^\gamma}{\gamma-1}   \right)} + \int_{\Gamma} \left(\frac{1}{2}|v_0|^2 + \frac{1}{2} |\Delta \eta_0|^2  \right)
\end{eqnarray*}
is the initial energy, we once again as in the previous case obtain $\|\eta\|_{L^\infty((0,T)\times\Gamma)}\leq H$. Let us first point out that if $\int_\Gamma F\geq 0$, the contact will break unconditionally due to $\eqref{blowup}$. Therefore let us assume that $\int_\Gamma F<0$. Then, the same calculation gives us
\begin{eqnarray*}
        \max_{t\in(0,T)} |\eta(t)|_\infty\geq && \left(\frac{Tm^\gamma}{|\Gamma|^{\gamma-1} C(1+\sqrt{T}) - |\Gamma|^{\gamma-1}\int_0^T \int_{\Gamma} F}\right)^{\frac1\gamma} \nonumber \\
		&&= \left(\frac{m^\gamma}{|\Gamma|^{\gamma-1} C\frac{(1+\sqrt{T})}{T} - |\Gamma|^{\gamma-1}\int_{\Gamma} F}\right)^{\frac1\gamma} \to \frac{m}{(-|\Gamma|^{\gamma-1}\int_\Gamma F)^{\frac1\gamma}} >H       
	\end{eqnarray*}
as $T\to \infty$, provided that
	\begin{eqnarray*}
		 \int_{\Gamma} F > -\frac{m^\gamma}{|\Gamma|^{\gamma-1}H^\gamma} \quad \text{ and } \quad \|F\|_{L^1(\Gamma)}\leq E_0,
	\end{eqnarray*}
which is a contradiction, implying $\eta(T)>0$ for some large enough $T$.

\section{Detachment of contact at a point by outer force - proof of Theorem $\ref{main3}$}
Denote $r = \sqrt{x^2+y^2}$. For a given $\kappa>0$, we first introduce a function 
\begin{eqnarray*}
    f_\kappa(x,y):= \begin{cases}
    \frac1{\pi\kappa^2}(\kappa-r),& \quad \text{ for } r\leq \kappa, \\
    0,& \quad \text{ otherwise},
    \end{cases}
\end{eqnarray*}
which satisfies
\begin{eqnarray*}
    \int_{\mathbb{R}^2} f_\kappa = \int_{\{r\leq \kappa\}} f_\kappa = 1.
\end{eqnarray*}
Now, let $(x_0,y_0) \in \Gamma$ and assume that $\eta(t,x_0,y_0)= 0$ for all $t\in (0,T)$ and let $F=f_\kappa^0 \sigma$ for $\sigma,\kappa>0$ with $f_\kappa^0(x,y):=f_\kappa(x-x_0,y-y_0)$. We calculate
\begin{eqnarray*}
    &&\int_0^t \int_\Gamma F \partial_t \eta = \int_\Gamma F \eta(t) - \int_\Gamma F \eta_0 \leq \int_\Gamma F \eta(t) \\
    &&\leq   \int_{B_\kappa(x_0,y_0)} f_\kappa^0\sigma C\|\Delta \eta(t)\|_{L^2(\Gamma)} \text{dist}((x_0,y_0),(x,y))^\alpha \\
    &&= C\sigma\|\Delta \eta(t)\|_{L^2(\Gamma)} \frac1{\pi}\int_0^{2\pi}\int_0^\kappa \left(\frac{1}{\kappa}- \frac{r}{\kappa^2} \right) r^\alpha \\
    &&=C\sigma\|\Delta \eta(t)\|_{L^2(\Gamma)} \frac2{(\alpha+1)(\alpha+2)}\kappa^{\alpha} \\
    &&\leq \frac12 \|\Delta \eta(t)\|_{L^2(\Gamma)}^2 + 2 \sigma^2 C^2 \frac{1}{(\alpha+1)^2(\alpha+2)^2} \kappa^{2\alpha},
\end{eqnarray*}
where $B_\kappa(x_0,y_0) = \{(x,y): \sqrt{(x-x_0)^2+(y-y_0)^2}\leq \kappa\}$, $C$ only depends on $\Gamma$ and $\alpha \in (0,1/2)$, where we have used the inequality $\eqref{max:est}$. Therefore, by choosing, say,
\begin{eqnarray*}
    \sigma^2 = \frac{(\alpha+1)^2(\alpha+2)^2}{2 C^2\kappa^{\alpha}}, \quad \kappa<1, \quad \alpha = 1/4,
\end{eqnarray*}
the energy is uniformly bounded and $F$ only depends on $\kappa$ so we can it as denote $F_\kappa$. Note that however $\int_\Gamma F_\kappa \to \infty$ as $\kappa\to 0$, which means that we can concentrate and increase the force near the point $(x_0,y_0)$ without increasing the energy. This happens because we assumed that the plate is "stuck" at $(x_0,y_0)$, meaning that applying stronger and stronger force near this point will not be able to break the contact.

\bigskip

By estimating as in Section $\ref{cont:sec}$, the contact inequality $\eqref{ver:ineq}$ for $t=T$ gives us 
\begin{eqnarray*}
    0<\int_\Gamma F_\kappa \leq \frac{C}T(1+\sqrt{T}).
\end{eqnarray*}
However, since $\int_\Gamma F_\kappa \to \infty$ as $\kappa\to 0$, we can choose $\kappa$ small enough such that this inequality is violated implying contradiction and consequently the existence of a $t\in (0,T)$ such that $\eta(t,x_0,y_0)>0$.

\appendix
    
\section{Some technical results}
Some of the following lemmas can be found in the literature. However, we choose to rewrite them in order to keep track on the dependence of constants on $\eta$, since this will be important in the calculation and since in our case the domain $\Omega^\eta$ can degenerate and even possibly consist of infinitely many disconnected components.

\begin{lem}\label{H1}
Let $\eta \in H^2(\Gamma)$ so that $\eta>0$ a.e. on $\Gamma$ and let $\bb{u}\in H_{0z}^1(\Omega^\eta)$ with $u_1,u_2\in H_0^1(\Omega^\eta)$. Then
	\begin{eqnarray*}
        \|\bb{u}\|_{H^1(\Omega^\eta)} \leq C(1+\|\eta\|_{L^\infty(\Gamma)}) \int_{\Omega^\eta} \mathbb{S}(\nabla \bb{u}):\nabla \bb{u},
	\end{eqnarray*}
where $C$ only depends on $\mu$.
\end{lem}
\begin{proof}
Due to $u_1,u_2\in H_0^1(\Omega^\eta)$, there is a sequence of $u_1^\omega,u_2^\omega \in C_c^\infty(\Omega^\eta)$ such that $u_1^\omega \to u_1$ and $u_2^\omega \to u_2$ in $H_0^1(\Omega^\eta)$ as $\omega\to 0$. Therefore, noticing that
    \begin{eqnarray*}
     &&\int_{\Omega^\eta} \partial_z u_2^\omega\partial_y u_3 = \int_{\Omega^\eta} \partial_z u_3\partial_y u_2^\omega, \quad   \int_{\Omega^\eta} \partial_z u_1^\omega\partial_x u_3 = \int_{\Omega^\eta} \partial_z u_3\partial_x u_1^\omega, \\
     &&\int_{\Omega^\eta} \partial_x u_2^\omega\partial_y u_1^\omega = \int_{\Omega^\eta} \partial_x u_1^\omega\partial_y u_2^\omega,
    \end{eqnarray*}
and denoting $\bb{u}^\omega:= (u_1^\omega,u_2^\omega, u_3)$, we can calculate
\begin{eqnarray*}
    &&\int_{\Omega^\eta} \mathbb{S}(\nabla \bb{u}^\omega):\nabla \bb{u}^\omega = \int_{\Omega^\eta} \Bigg[\mu \left( \nabla \bb{u}^\omega + \nabla^\tau \bb{u}^\omega- \frac23\nabla \cdot \bb{u}^\omega \mathbb{I} \right) +\lambda\nabla \cdot \bb{u}^\omega \mathbb{I} \Bigg]:\nabla \bb{u}^\omega\\
    &&= \mu \|\nabla \bb{u}^\omega \|_{L^2(\Omega^\eta)}^2 + \left(\frac\mu3+\lambda \right)\|\nabla \cdot\bb{u}^\omega \|_{L^2(\Omega^\eta)}^2,
\end{eqnarray*}
which implies
    \begin{eqnarray*}
         \|\nabla \bb{u}^\omega\|_{L^2(\Omega^\eta)}^2 \leq C \int_{\Omega^\eta} \mathbb{S}(\nabla \bb{u}^\omega):\nabla \bb{u}^\omega,
    \end{eqnarray*}
where $C$ only depends on $\mu$, so by passing to the limit $\omega\to 0$ the following holds
	\begin{eqnarray*}
		\|\nabla \bb{u}\|_{L^2(\Omega_E^\eta)}^2 \leq C \int_{\Omega^\eta} \mathbb{S}(\nabla \bb{u}):\nabla \bb{u}.
	\end{eqnarray*}
On the other hand, for a.a. $(x,y)\in \Gamma$
	\begin{eqnarray*}
		|u_i(x,y,z)|^2 = \left|\int_0^z \partial_z u_i(x,y,s) ds \right|^2\leq  \left|\int_0^{\eta(x,y)} |\partial_z u_i(x,y,s)| ds \right|^2 \leq |\eta|_\infty \int_0^{\eta(x,y)} |\partial_z u_i(x,y,s)|^2,
	\end{eqnarray*}
so integrating over $\Omega^\eta$
	\begin{eqnarray*}
		\int_{\Omega^\eta}|\bb{u}|^2   \leq |\eta|_\infty^2 \int_{\Omega^\eta} |\partial_z \bb{u}|^2 \leq C |\eta|_\infty^2\int_{\Omega^\eta} \mathbb{S}(\nabla \bb{u}):\nabla \bb{u}.
	\end{eqnarray*}
Combining the above inequalities, the proof is finished.
\end{proof}

 Next, it will be useful to extend fluid velocity by zero to a larger connected domain:

\begin{lem}\label{ext:zero}
Let $\eta \in H^2(\Gamma)$ so that $\eta>0$ a.e. on $\Gamma$ and let $\bb{u}\in H_{0z}^1(\Omega^\eta)$. Then, extension by zero of $\bb{u}$ to $\Omega_E^\eta:= \Omega^\eta \cup \Gamma\times (-1,0)$ defined as 
    \begin{eqnarray*}
        e_0[\bb{u}]:= 
        \begin{cases}
        \bb{u},& \text{ on } \Omega^\eta, \\
         0, & \text{ on } \Gamma\times (-1,0),
         \end{cases}
    \end{eqnarray*}
satisfies
	\begin{eqnarray*}
        \|\bb{u}\|_{H^1(\Omega^\eta)}=\|e_0[\bb{u}]\|_{H^1(\Omega_E^\eta)}.
	\end{eqnarray*}
\end{lem}
\begin{proof}
Let $O\subset \Omega_E^\eta$ be an open set and let $\varphi \in C_c^\infty(O)$. Obviously, if $O\subset \Omega^\eta$ then
	\begin{eqnarray*}
		\int_O \nabla e_0[\bb{u}] \varphi = \int_O \nabla \bb{u} \varphi
	\end{eqnarray*}
and if $O \subset \Gamma\times (-1,0)$ then
	\begin{eqnarray*}
		\int_O \tilde{\nabla} \bb{u} \varphi = 0.
	\end{eqnarray*}
If $O \cap \Gamma\times \{0\}\neq \emptyset$, then
	\begin{eqnarray*}
		&&-\int_O e_0[\bb{u}] \cdot\nabla\varphi = \int_{O\cap \{z>0\}} \nabla e_0[\bb{u}] \varphi  + \underbrace{\int_{O\cap \{z<0\}} \nabla e_0[\bb{u}] \varphi}_{=0} - \int_{O \cap \Gamma\times \{0\}}\underbrace{\gamma_{|O\cap \{z=0\}} e_0[\bb{u}]}_{=\gamma_{|O\cap \{z=0\}} \bb{u}=0}\cdot\bb{e}_z \varphi \\
		&&=\int_{O\cap \{z>0\}} \nabla e_0[\bb{u}] \varphi = \int_{O\cap \{z>0\}} \nabla \bb{u} \varphi,
	\end{eqnarray*}
so $\nabla e_0[\bb{u}] \in L^2(\Omega_E^\eta)$ and 
	\begin{eqnarray*}
		\nabla e_0[\bb{u}] = \begin{cases}
  \nabla\bb{u},& \text{ on } \Omega^\eta, \\
  0, & \text{ on } \Gamma\times (-1,0).
 \end{cases}
    \end{eqnarray*}
The proof is thus finished.
\end{proof}

\begin{lem}\label{lem:etat}
Let $\eta \in H^2(\Gamma)$ so that $\eta>0$ a.e. on $\Gamma$ and let $\bb{u}\in H_{0z}^1(\Omega)$. Then
\begin{eqnarray*}
    \int_{\Gamma} \frac{|\gamma_{|\Gamma^\eta}\bb{u}|^2}{\eta} \leq \int_{\Omega^\eta} |\partial_z \bb{u}|^2, \\
     \int_{\Omega^\eta} \frac{|\bb{u}|^2}{\eta^2} \leq \int_{\Omega^\eta} |\partial_z \bb{u}|^2.
\end{eqnarray*}
\end{lem}
\begin{proof}
For a.a. $(x,y)\in \Gamma$ we have $\eta(x,y)>0$ and $\int_{0}^{\eta(x,y)} \partial_z\bb{u} (x,y,z)dz = \gamma_{|\Gamma^\eta}\bb{u}(x,y)$ so we can estimate
	\begin{eqnarray*}
		&&|\gamma_{|\Gamma^\eta}\bb{u}(x,y)| =\left|\int_{0}^{\eta(x,y)} \partial_z \bb{u}(x,y,z) dz\right| \\
		&&\leq \left(\int_{0}^{\eta(x,y)} |\partial_z e_0[\bb{u}](x,y,z)|^2dz \right)^{1/2}\left(\int_{0}^{\eta(x,y)}1dz \right)^{1/2}\\
 &&= \left(\int_{0}^{\eta(x,y)} |\partial_z e_0[\bb{u}](x,y,z)|^2dz\right)^{1/2} \sqrt{\eta(x,y)}
	\end{eqnarray*}
and thus 
	\begin{eqnarray*}
		\frac{|\gamma_{|\Gamma^\eta}\bb{u}(x,y)|^2}{\eta(x,y)}  \leq  \int_{0}^{\eta(x,y)} |\partial_z \bb{u}(x,y,z)|^2dz.
	\end{eqnarray*}
Since this holds for a.a. $(x,y)\in \Gamma$, we can integrate over $\Gamma$ gives us the first inequality. Next, once again for a.a. $(x,y)\in \Gamma$ and a.a. $0<z<\eta(x,y)$ one has similarly
\begin{eqnarray*}
    \frac{|u(x,y,z)|^2}{\eta(x,y)} \leq  \int_{0}^{\eta(x,y)} |\partial_z \bb{u}(x,y,s)|^2ds,
\end{eqnarray*}
so
\begin{eqnarray*}
    \int_0^{\eta(x,y)}\frac{|u(x,y,z)|^2}{|\eta(x,y)|^2}dz \leq   \int_0^{\eta(x,y)}\frac{1}{\eta(x,y)}\left[\int_{0}^{\eta(x,y)} |\partial_z \bb{u} (x,y,s)|^2ds\right]dz = \int_{0}^{\eta(x,y)} |\partial_z \bb{u} (x,y,s)|^2ds.
\end{eqnarray*}
Integrating over $\Gamma$, the second inequality follows.
\end{proof}

\begin{lem}\label{imbed}
Let $\eta \in H^2(\Gamma)$ so that $\eta>0$ a.e. on $\Gamma$ and let $\bb{u}\in H_{0z}^1(\Omega^\eta)$. Then, we have
		\begin{eqnarray*}
             \|\bb{u}\|_{L^p(\Omega^\eta)}\leq C\|\bb{u}\|_{H^1(\Omega^\eta)}.
		\end{eqnarray*}
for any $p\in [1,6)$, where $C=C(\|1+\eta\|_{L^\infty(\Gamma)}, \|\Delta \eta\|_{L^2(\Gamma)},\Gamma,p)$.
\end{lem}
\begin{proof}
By Lemma $\ref{ext:zero}$, we can extend $\bb{u}$ by zero to $\Omega_E^\eta:= \Omega^\eta \cup \Gamma\times (-1,0)$ and denote by $e_0[\bb{u}]$, so we have $\|\bb{u}\|_{H^1(\Omega^\eta)}=\|e_0[\bb{u}]\|_{H^1(\Omega_E^\eta)}$. Then, we can map $e_0[\bb{u}]$ onto $\Omega_E:=\Gamma\times (-1,0)$ by coordinate transform
	\begin{eqnarray*}
        A_\eta:&&\Omega_E\to\Omega_E^\eta\nonumber\\
         &&(x,y,z)\mapsto(x,y, (z+1)\eta(x,y)+z).
	\end{eqnarray*}
Denoting $\hat{\bb{u}}(x,y,z):=e_0[\bb{u}](A_\eta(x,y,z))$, the Jacobain of transformation as $J:=1+\eta$ and the gradient on $\Omega_E$ as $\hat{\nabla}$, we have
		\begin{eqnarray*}
             &&\int_{\Omega_E}|\hat{\nabla}\hat{\bb{u}}|^q \leq \int_{\Omega_E^\eta}\frac1J J^{q}|\nabla e_0[\bb{u}]|^q |\nabla \eta|^q \leq 
            \|1+\eta \|_{L^\infty(\Gamma)}^{q-1} \|\nabla e_0[\bb{u}]\|_{L^2(\Omega_E^\eta)}^{q}\|\nabla\eta\|_{L^{\frac{q^2}{2-q}(\Omega_E^\eta)}}^q  \\
             &&\leq C\underbrace{\|\nabla e_0[\bb{u}]\|_{L^2(\Omega_E^\eta)}^{q}}_{ =\|\nabla \bb{u}\|_{L^2(\Omega^\eta)}^{q} } \|1+\eta\|_{L^\infty(\Gamma)}^{\frac{q^2-2q+2}{q}} \|\Delta\eta\|_{L^2(\Gamma)}^q,
		\end{eqnarray*}
for $q\in [1,2)$ with $C=C(q,\Gamma)$, where we used
	\begin{eqnarray*}
		&&\|\nabla\eta\|_{L^{\frac{q^2}{2-q}(\Omega_E^\eta)}} = \left(\int_{\Omega_E^\eta} |\nabla \eta|^{\frac{q^2}{2-q}} \right)^{\frac{2-q}{q^2}} = \left((1+\eta)\int_{\Gamma} |\nabla \eta|^{\frac{q^2}{2-q}} \right)^{\frac{2-q}{q^2}}\\
		&&\leq \|1+\eta\|_{L^\infty(\Gamma)}^{\frac{2-q}{q^2}} \|\nabla\eta\|_{L^{\frac{q^2}{2-q}(\Gamma)}} \leq C\|1+\eta\|_{L^\infty(\Gamma)}^{\frac{2-q}{q^2}} \|\Delta \eta\|_{L^2(\Gamma)}.
	\end{eqnarray*}
On the other hand
	\begin{eqnarray*}
		\int_{\Omega_E} |\hat{\bb{u}}|^q = \int_{\Omega_E^\eta}\frac1J |e_0[\bb{u}]|^q \leq\int_{\Omega_E^\eta}|e_0[\bb{u}]|^q, 
	\end{eqnarray*}
so combining the above inequalities we have
		\begin{eqnarray}
             \|\hat{\bb{u}}\|_{W^{1,q}(\Omega_E)}\leq C \left(1+\|1+\eta\|_{L^\infty(\Gamma)}^{\frac{q^2-2q+2}{q^2}}\|\Delta \eta\|_{L^2(\Gamma)} \right) \|\bb{u}\|_{H^1(\Omega^\eta)}.   \label{ineq:w1q}
		\end{eqnarray}
Finally, by Sobolev imbedding on $\Omega_E$
	\begin{eqnarray*}
		&&\left(\int_{\Omega_E^\eta} |e_0[\bb{u}]|^p \right)^{\frac1p}=\left(\int_{\Omega_E} J |\hat{\bb{u}}|^p \right)^{\frac1p} \leq \|1+\eta\|_{L^\infty(\Gamma)}^{\frac1p} \|\hat{\bb{u}}\|_{L^p(\Omega_E)} \leq C \|1+\eta\|_{L^\infty(\Gamma)}^{\frac1p} C \|\hat{\bb{u}}\|_{W^{1,q}(\Omega_E)} \\
		&& \leq C\|1+\eta\|_{L^\infty(\Gamma)}^{\frac1p}\left(1+(\|1+\eta\|_{L^\infty(\Gamma)}^{\frac{p^2-2p+2}{p^2}})\|\Delta \eta\|_{L^2(\Gamma)} \right) \|\bb{u}\|_{H^1(\Omega^\eta)}
	\end{eqnarray*}
for any $q\in[1,2)$ and $p=\frac{3q}{3-q}\in[1,6)$, so the proof is finished.
\end{proof}

\begin{lem}\label{trace:lemma}
Let $\eta \in H^2(\Gamma)$ so that $\eta>0$ a.e. on $\Gamma$ and let $\bb{u}\in H_{0z}^1(\Omega^\eta)$. Then, one has
	\begin{eqnarray}
  \|\gamma_{|\Gamma^\eta} \bb{u}\|_{H^s(\Gamma)} \leq C \| \bb{u}\|_{H^1(\Omega^\eta)} \label{trace1}
	\end{eqnarray}
for any $s\in (0,1/2)$, where $C=C(\|1+\eta\|_{L^\infty(\Gamma)}, \|\Delta \eta\|_{L^2(\Gamma)},\Gamma,s)$, and consequently
	\begin{eqnarray}
  \|\gamma_{|\Gamma^\eta} \bb{u}\|_{L^p(\Gamma)} \leq C \| \bb{u}\|_{H^1(\Omega^\eta)} \label{trace2}
	\end{eqnarray}
for any $p\in [1,4)$, where $C=C(\|1+\eta\|_{L^\infty(\Gamma)}, \|\Delta \eta\|_{L^2(\Gamma)},\Gamma,p)$.	
\end{lem}
\begin{proof}
Similarly as in Lemma $\ref{imbed}$, we can extend $\bb{u}$ by zero to $\Omega_E^\eta$ and denote by $e_0[\bb{u}]$, and then map onto a fixed domain $\Omega_E$ and denote the pull-back by $\hat{\bb{u}}$. Then, since $\hat{\bb{u}} \in W^{1,p}(\Omega_E)$ for any $p\in [1,2)$, the trace $\gamma_{|\{z=0\}} \hat{\bb{u}}$ satisfies
	\begin{eqnarray*}
		\|\gamma_{|\{z=0\}} \hat{\bb{u}}\|_{H^s(\Gamma)} \leq C \| \hat{\bb{u}}\|_{W^{1,p}(\Omega^\eta)}\leq C \|\bb{u}\|_{H^1(\Omega^\eta)}
	\end{eqnarray*}
by $\eqref{ineq:w1q}$, for any $s\in(0,1/2)$ and $p=p(s)<2$, where $C=C(\|1+\eta\|_{L^\infty(\Gamma)}, \|\Delta \eta\|_{L^2(\Gamma)},\Gamma,s)$. Since $\gamma_{|\{z=0\}} \hat{\bb{u}} = \gamma_{|\Gamma^\eta}\bb{u}=\gamma_{|\Gamma^\eta}e_0[\bb{u}]$ by definition, $\eqref{trace1}$ follows. Now, inequality $\eqref{trace2}$ follows directly by Sobolev imbedding.
\end{proof}
	
The following lemma follows the proof of \cite[Proposition 12]{global}, where similar result was obtained. However, here we generalize it to a 2D case and push the exponent to $\alpha = 1/2$ (i.e. we choose $b=\sqrt{\eta}$) to obtain what seems to be the optimal result:

\begin{lem}\label{eta:n}
Let $s\geq 1$ and let $\eta\in W^{1,2s}(\Gamma)\cap W^{2,s}(\Gamma)$ such that $\eta > 0$. Then
	\begin{eqnarray*}
  \int_\Gamma\frac{|\nabla \eta|^{2s}}{\eta^s} \leq C\int_\Gamma |\nabla^2 \eta|^s,
	\end{eqnarray*}
where $C$ only depends on $s$.
\end{lem}\begin{proof}
Denote $b=\sqrt{\eta}$. By partial integration
	\begin{eqnarray*}
		&&0=\int_\Gamma \nabla\cdot(b \nabla b |\nabla b|^{2s-2}) = \int_\Gamma |\nabla b|^{2s} + \int_\Gamma  b \Delta b|\nabla b|^{2s-2}+ \int_\Gamma b\nabla b\cdot \nabla(|\nabla b|^{2s-2}).
	\end{eqnarray*}
First we calculate
	\begin{eqnarray*}
		\nabla(|\nabla b|^{2s-2}) = \nabla (|\nabla b|^2)^{s-1} = (s-1) (|\nabla b|^2)^{s-2} \nabla |\nabla b|^2,
	\end{eqnarray*}
and
    \begin{eqnarray*}
		\nabla b\cdot \nabla|\nabla b|^2 = 2\begin{bmatrix} b_x \\ b_y \end{bmatrix} \begin{bmatrix}
		b_{xx} b_x + b_{xy}b_x \\ b_{xy} b_y + b_yb_{yy} 
		\end{bmatrix} =2|\nabla b|^2 \Delta b + (4 b_x b_y b_{xy} - 2(b_x)^2 b_{yy} - 2 (b_y)^2 b_{xx} ),
	\end{eqnarray*}
so
	\begin{eqnarray*}
		\nabla b\cdot  \nabla(|\nabla b|^{2s-2}) = (s-1) (|\nabla b|^2)^{s-2} \left( 2|\nabla b|^2 \Delta b + (4 b_x b_y b_{xy} - 2(b_x)^2 b_{yy} - 2 (b_y)^2 b_{xx} )\right).
		\end{eqnarray*}
Therefore
	\begin{eqnarray*}
		0&=&  \int_\Gamma |\nabla b|^{2s}+ (2s-1) \int_\Gamma  b \Delta b|\nabla b|^{2s-2} \\
        &&+ (2s-2)\int_\Gamma b\left( 2 b_x b_y b_{xy} - (b_x)^2 b_{yy} -  (b_y)^2 b_{xx} \right)|\nabla b|^{2s-4}.
	\end{eqnarray*}
Expressing $b$, one has
	\begin{eqnarray*}
		\int_\Gamma |\nabla b|^{2s} =  \frac1{4^s}\int_\Gamma \frac{|\nabla \eta|^{2s}}{\eta^s}    
	\end{eqnarray*}
then
	\begin{eqnarray*}
		&&\int_\Gamma  b \Delta b|\nabla b|^{2s-2}\\
        &&= \int_\Gamma \sqrt{\eta} \left(\frac12 \frac{\Delta\eta}{\sqrt{\eta}} - \frac14\frac{|\nabla\eta|^2}{\eta^{3/2}} \right) \frac1{4^{s-1}}\frac{|\nabla \eta|^{2s-2}}{\eta^{s-1}} = \frac1{2^{2s-1}}\int_\Gamma \Delta\eta \frac{|\nabla \eta|^{2s}}{\eta^s} - \frac1{4^s}\int_\Gamma \frac{|\nabla \eta|^{2s}}{\eta^s}
	\end{eqnarray*}
and finally
	\begin{eqnarray*}
		&&\int_\Gamma b\left( 2 b_x b_y b_{xy} - (b_x)^2 b_{yy} -  (b_y)^2 b_{xx} \right)|\nabla b|^{2s-4} \\
         &&=\int_{\Gamma} \frac{1}{16\eta}\left(  2\eta_x \eta_y \eta_{xy}- \eta_x^2 \eta_{yy} - \eta_y^2 \eta_{xx} \right)\frac1{4^{s-2}}\frac{|\nabla \eta|^{2s-4}}{\eta^{s-2}},
	\end{eqnarray*}
which combined together imply
	\begin{eqnarray*}
		&&\left(\frac1{4^s} +\frac{2s-1}{16} \right)\int_\Gamma \frac{|\nabla \eta|^{2s}}{\eta^s} \\
		&&=  \frac{2s-1}{2^{2s-1}}\int_\Gamma \Delta\eta \frac{|\nabla \eta|^{2s-2}}{\eta^{s-1}} -\frac{2s-2}{4^s}\int_{\Gamma} \frac{1}{\eta}\left(  2\eta_x \eta_y \eta_{xy}- \eta_x^2 \eta_{yy} - \eta_y^2 \eta_{xx} \right)\frac{|\nabla \eta|^{2s-4}}{\eta^{s-2}}\\
		&&\leq \frac1{4^s}\int_\Gamma \frac{|\nabla \eta|^{2s}}{\eta^s} + C(s) \int_\Gamma |\nabla^2 \eta|^s,
	\end{eqnarray*}
where we used the Young inequality, so the proof is finished.

\end{proof}
\begin{cor}
Let $s\geq 1$ and let $\eta\in W^{1,2s}(\Gamma)\cap W^{2,s}(\Gamma)$ such that $\eta > 0$ a.e. on $\Gamma$. Then
	\begin{eqnarray*}
		\int_{\Gamma}\frac{|\nabla \eta|^{2s}}{\eta^s} \leq C\int_\Gamma |\nabla^2 \eta|^s,
	\end{eqnarray*}
where $C$ only depends on $s$.
\end{cor}
\begin{proof}
Let $\delta>0$. Then, by previos lemma
	\begin{eqnarray*}
		\int_\Gamma\frac{|\nabla \eta|^{2s}}{(\eta+\delta)^s} \leq C\int_\Gamma |\nabla^2 \eta|^s.
	\end{eqnarray*}
Since $\{\eta = 0\}$ is of measure zero, we have that $\frac{|\nabla \eta|^{2s}}{(\eta+\delta)^s} \to \frac{|\nabla \eta|^{2s}}{\eta^s}$ a.e. on $\Gamma$. Moreover, as $\frac{|\nabla \eta|^{2s}}{(\eta+\delta)^s} < \frac{|\nabla \eta|^{2s}}{\eta^s}$ a.e. on $\Gamma$, we have for any $\varepsilon>0$
\begin{eqnarray*}
    \int_{\{\eta\geq \varepsilon\}}\frac{|\nabla \eta|^{2s}}{\eta^s}= \lim_{\delta\to 0} \int_{\{\eta\geq \varepsilon\}}\frac{|\nabla \eta|^{2s}}{(\eta+\delta)^s} \leq \lim_{\delta\to 0} \int_{\Gamma}\frac{|\nabla \eta|^{2s}}{(\eta+\delta)^s} \leq C\int_\Gamma |\nabla^2 \eta|^s
\end{eqnarray*}
by Lebesque dominated convergence theorem. Passing to the limit $\varepsilon \to 0$, the proof is complete. 
\end{proof}

\begin{lem}\label{ln:lemma}
Let $\eta_n \to \eta$ in $C([0,T]\times \Gamma)$ as $n\to\infty$ and assume that $\ln\eta_n$ is uniformly bounded in $C([0,T]; L^1(\Gamma))$. Then, for all $t\in [0,T]$ and any non-negative $\varphi\in C^\infty(\Gamma)$, one has
\begin{eqnarray}
    \int_\Gamma \ln \eta(t)^-\varphi \leq \lim_{n\to\infty} \int_\Gamma\ln\eta_n(t)^- \varphi. \label{ln:ineq}
\end{eqnarray}
Moreover, $\ln\eta \in C([0,T]; L^1(\Gamma))$.
\end{lem}
\begin{proof}
Let $\delta\in(0,1)$ and fix $t\in[0,T]$. Then, since $x\mapsto \ln(x+\delta)^-$ is continuous on $[0,\infty)$, we have that $\ln(\eta_n+\delta)^-\to \ln (\eta+\delta)^-$ in $C([0,T]\times \Gamma)$ as $\delta\to 0$ and consequently
\begin{eqnarray*}
    \ln (\eta(t)+\delta)^- = \lim_{n\to\infty} \ln(\eta_n(t)+\delta)^- \quad \text{ on } \Gamma.
\end{eqnarray*}
Therefore, since $\ln(x+\delta)^- \leq \ln(x)^-$ on $[0,\infty)$, one has
\begin{eqnarray*}
    \int_\Gamma \ln (\eta(t)+\delta)^-\varphi = \lim_{n\to\infty}  \int_\Gamma\ln(\eta_n(t)+\delta)^- \varphi\leq \lim_{n\to\infty}  \int_\Gamma\ln\eta_n(t)^- \varphi,
\end{eqnarray*}
for any non-negative $\varphi\in C^\infty(\Gamma)$. Now, assume there is a set of positive measure $S\subset \Gamma$ such that $\eta(t)=0$ on $S$ and let $\varphi>0$ on $S$. Then,
\begin{eqnarray*}
   |\ln \delta|\int_S \varphi  \leq  \int_S \ln (\eta(t)+\delta)^-\varphi \leq  \int_\Gamma \ln (\eta(t)+\delta)^-\varphi \leq \lim_{n\to\infty}  \int_\Gamma\ln\eta_n(t)^- \varphi,
\end{eqnarray*}
so the left-hand side blows up as $\delta\to 0$ which is a contradiction. Therefore, $\eta(t)=0$ on a set of measure zero and consequently $\ln(\eta(t)+\delta)^-\to \ln\eta(t)^-$ a.e. on $\Gamma$. By the Theorem of Egorov, for every $\varepsilon>0$ there exists $A_\varepsilon\subset \Gamma$ such that $|\Gamma\setminus A_\varepsilon|<\varepsilon$ and $\ln(\eta(t)+\delta)^-\to \ln(\eta(t))^-$ uniformly on $A_\varepsilon$, so
\begin{eqnarray*}
    \lim_{\delta\to 0}\int_{A_\varepsilon} \ln (\eta(t)+\delta)^-\varphi =\int_{A_\varepsilon} \ln \eta(t)^-\varphi \leq 
    \lim_{n\to \infty} \int_\Gamma\ln\eta_n(t)^- \varphi.
\end{eqnarray*}
Passing to the limit $\varepsilon\to 0$, we obtain $\eqref{ln:ineq}$. \\

Now, let $[0,T]\ni t_m \to t$ as $m\to\infty$ and let $\kappa>0$. Then
\begin{eqnarray*}
    &&\int_\Gamma |\ln\eta(t)^--\ln\eta(t_m)^-| \\
    &&\leq   \int_\Gamma |\ln(\eta(t))^--\ln(\eta(t)+\kappa)^-|+\int_\Gamma |\ln(\eta(t)+\kappa)^--\ln(\eta(t_m)+\kappa)^-|\\
    &&\quad+\int_\Gamma |\ln(\eta(t_m)+\kappa)^--\ln(\eta(t_m))^-| =: I_1 + I_2 + I_3.
\end{eqnarray*}
Since $\ln(\eta(t))^- \in L^1(\Gamma)$ and since $\ln(\eta(t)+\kappa)^- \leq \ln(\eta(t))^-$, we have by the Lebesque dominated convergence theorem that
\begin{eqnarray*}
    \ln(\eta(t)+\kappa)^- \to \ln(\eta(t))^- \quad \text{and} \quad \ln(\eta(t_m)+\kappa)^- \to \ln(\eta(t_m))^- \quad \text{as }\kappa \to 0
\end{eqnarray*}
where the second convergence follows in the same way. Therefore, for any $a>0$, there is a $\kappa>0$ small enough such that $I_1 + I_3 \leq \frac{2a}3$. Next, since $t\mapsto \int_\Gamma\ln(\eta(t)+\kappa)^-\in C([0,T])$, there is an $m_0$ large enough such that for all $m\geq m_0$ one has $I_2 \leq \frac{a}3$. Therefore, one concludes that 
\begin{eqnarray*}
   \lim_{m\to \infty}\int_\Gamma |\ln\eta(t)^--\ln\eta(t_m)^-| =0.
\end{eqnarray*}
Since $\{t_m\}_{m\in \mathbb{N}}$ was arbitrary, the proof is finished.

\end{proof}

\noindent
\textbf{Acknowledgment}. This research was supported by the Science Fund of the Republic of Serbia, GRANT No TF C1389-YF, Project title - FluidVarVisc. The author would like to express his gratitude to Boris Muha for his useful comments and suggestions, which in particular gave rise to the second case in Theorem \ref{main2}.

	
\end{document}